\documentclass[12pt]{article}
 \usepackage[body={6.5in,9.0in},left=1in,top=1in,centering]{geometry}
\usepackage{graphicx}
\usepackage{amssymb}
\usepackage{epstopdf,amsmath,amsthm,dsfont,amsfonts,mathrsfs}
\usepackage{color,pstricks,graphicx}
\usepackage[colorlinks=true,urlcolor=blue,linkcolor=orange,citecolor=blue,pdfstartview=FitH]{hyperref}
\newtheorem{theorem}{Theorem}
\newtheorem{lemma}{Lemma}
\newtheorem{assumption}{Assumption}
\newtheorem{proposition}{Proposition}
\newtheorem{corollary}{Corollary}
\theoremstyle{definition}
\newtheorem{definition}{Definition}
\newtheorem{remark}{Remark}
\newtheorem{example}{Example}

\newcommand{\R}{\mathbb{R}}

\newcommand{\E}{\mathbb{E}}

\newcommand{\e}{\varepsilon}

\newcommand{\ba}{\bar{\alpha}}
\newcommand{\cP}{\mathcal P}
\newcommand{\bW}{\mathbb W}
\newcommand{\wdt}{\widetilde}

\newcommand{\tr}{\mathrm{Tr}}
\newcommand{\diag}{\mathrm{diag}}
\newcommand{\comment}[1]{}
\allowdisplaybreaks
\DeclareMathOperator*{\argmin}{arg\,min} 
\title{Ergodic McKean-Vlasov Games: Verification Theorems and Linear-Quadratic Applications}

\author{Qingshuo Song\thanks{Department of Mathematical Sciences, Worcester Polytechnic Institute, \url{qsong@wpi.edu}. }
\and
    Gu Wang\thanks{Department of Mathematical Sciences, Worcester Polytechnic Institute, \url{gwang2@wpi.edu} 
    This author is partially supported by NSF grant DMS-2206282.}
\and
    Zuo Quan Xu\thanks{Department of Applied Mathematics, The Hong Kong Polytechnic University, \url{maxu@polyu.edu.hk}}
  \and
    Chao Zhu\thanks{Department of Mathematical Sciences, University of Wisconsin-Milwaukee, Milwaukee, WI 53201. \url{zhu@uwm.edu}. The research of this author was supported in part by the Simons foundation under grant number 8035009.}
}

\begin{document}
\maketitle
\begin{abstract}
    This paper investigates two-player ergodic nonzero-sum stochastic differential games with McKean-Vlasov dynamics. 
    We establish a verification theorem connecting solutions of coupled Hamilton-Jacobi-Bellman (HJB) Master equations to Nash equilibria, 
    characterized through an auxiliary control problem defined on the measure space. 
    A key contribution is showing that the value functions are uniquely determined (up to an additive constant) by the uniqueness of the invariant measure of the optimal state process. 
    The theory is applied to Linear-Quadratic-Gaussian (LQG) settings, where explicit solutions to the Master equations are derived by exploiting their polynomial structure in measure variables. 
    
    
\end{abstract}

\section{Introduction}\label{sec:intro}

Stochastic control theory and differential games have been extensively studied and applied across diverse fields including economics, engineering, and finance \cite{FR75, YZ99, JN18}.
Two important generalizations are ergodic control problems, which optimize long-term average cost over an infinite horizon \cite{ABG12, JJSY24}, and stochastic differential games \cite{AB07, ED81, GKPP23}.
More recently, mean-field games and mean-field control problems have emerged as powerful frameworks for analyzing systems with large numbers of agents \cite{CD18I, CD18II}.

However, to the best of our knowledge, stochastic differential games 
combining ergodic criteria with mean-field dynamics have not yet been analyzed in the literature.
This paper aims to bridge this gap by investigating two-player ergodic nonzero-sum differential
games under McKean-Vlasov dynamics.
To illustrate the key concepts, we consider a simplified model with a pair of 
controlled diffusion processes $X_t = (X_{1,t}, X_{2,t})$:
\begin{equation}
    \label{eq:X_intro}
        dX_t = \alpha_t dt + d W_t, 
\end{equation}
where $W_t$ is a standard two-dimensional Brownian motion, 
and $\alpha_t = (\alpha_{1,t}, \alpha_{2,t})$ denotes 
the control processes chosen by the two players.
The cost functional for player $i$ ($i=1,2$) is given by
\begin{equation}
    \label{eq:J_intro}
    \hat J_i(\alpha_1, \alpha_2)
    = \lim_{T \to \infty} \frac{1}{T}
    \mathbb E \left[ \int_0^T \left(
        \gamma \mathbb E [ |X_t|^2 ] + (1 - \gamma) |X_{t}|^2
        + r_i |\alpha_{i,t}|^2 + 
        \eta_i \left( \mathbb E[X_{i,t}] \right)^2
    \right) dt \right],
\end{equation}
where $\gamma \in [0,1]$, and $r_i, \eta_i \ge 0$ are given constants.
The objective is to identify a Nash equilibrium
$\alpha^* = (\alpha_1^*, \alpha_2^*)$ satisfying
\begin{equation}
    \label{eq:Nash_intro}
    \hat c_1 := \hat J_1(\alpha_1^*, \alpha_2^*) \le \hat J_1(\alpha_1, \alpha_2^*), \quad
    \hat c_2 := \hat J_2(\alpha_1^*, \alpha_2^*) \le \hat J_2(\alpha_1^*, \alpha_2),
\end{equation}        
for all admissible controls $\alpha_1$ and $\alpha_2$. A rigorous formulation in a more 
general framework is provided in Section \ref{sec:formulation}.

It is worth noting that the cost function in \eqref{eq:J_intro}
incorporates both the state process $X_t$ and its distribution $\mu_t := \mathcal L(X_t)$,
which is a common feature in McKean-Vlasov control problems and mean-field games;
see, e.g., \cite{CD18I,CD18II,CDLL19,Car19} and the references therein.
The presence of distribution-dependent terms in the cost function introduces significant complexity, 
as it necessitates solving for distribution-dependent optimal controls. 
Consequently, the state dynamics for $X_t$ are naturally described by McKean-Vlasov type SDEs,
where the drift and diffusion coefficients depend on both the state and its distribution.

Our approach to solving such a game problem in a general framework 
can be tied up to the study of the connection between the game problem 
and a system of coupled
Hamilton-Jacobi-Bellman (HJB) type Master equations of the form 
\begin{equation}
    \label{eq:master_intro1}
    \int_{\mathbb R^2} 
        F_i^\gamma \left(\mu, x, D \frac{\delta v_i}{\delta \mu}(\mu, x), 
        D^2 \frac{\delta v_i}{\delta \mu}(\mu, x), D \frac{\delta v_{3-i}}{\delta \mu}(\mu, x) 
        \right) 
         \mu(dx) = c_i, \quad i=1,2,
\end{equation}
where $F_i^\gamma$ is a fully nonlinear function, and $\frac{\delta v}{\delta \mu}$ denotes
the flat derivative with respect to the measure argument; see \cite{CDLL19}
for details. The objective of the equation system \eqref{eq:master_intro1} is to find 
the solution quadruple $(v_1, v_2, c_1, c_2)$, where $v_i(\mu)$
is an infinite-dimensional function on a measure space 
and $c_i \in \mathbb R$ is a constant for $i=1,2$.

The major theoretical contribution of this paper is to establish a verification theorem 
that connects the solution of the Master equations \eqref{eq:master_intro1} to the stochastic 
differential game problem in a general framework, which covers 
the above game problem \eqref{eq:X_intro}--\eqref{eq:Nash_intro}.
First, we prove that the constant $c_i$ in the Master equation
coincides with the ergodic constant $\hat c_i$ in the game problem under suitable conditions. 
We further show that the function $v_i$ can be interpreted as the value function 
of an auxiliary control problem, which is defined on the measure space
and is associated with the original game problem.
We emphasize that, unlike in standard control problems, 
$v_i$ cannot be uniquely identified as the value function, 
since the Master equations \eqref{eq:master_intro1} are invariant 
under constant shifts. 
Specifically, if $(v_1, v_2, c_1, c_2)$ is a solution, 
so is $(v_1 + K_1, v_2 + K_2, c_1, c_2)$ 
for any constants $K_1, K_2 \in \mathbb R$.
To resolve this non-uniqueness, we require the uniqueness of the invariant measure
of the optimal state process as an additional condition, which pins down the solution 
to the Master equations \eqref{eq:master_intro1} for a carefully chosen value $K_i$,
which is distinguished from the existing literature.

Another major contribution of this paper is
Section \ref{sec:LQG} to provide explicit solutions 
to the Master equations \eqref{eq:master_intro1} 
in the Linear-Quadratic-Gaussian (LQG) setting,
which includes the game problem \eqref{eq:X_intro}--\eqref{eq:Nash_intro} as a special case.
Indeed, Master equations are generally intractable 
due to their infinite-dimensional nature and fully nonlinear structure.
Our approach to solving the Master equations \eqref{eq:master_intro1} explicitly
relies on the polynomial structure of the cost function $\ell_i$ on measure spaces;
see Section \ref{sec:flatderivative} for the definition of polynomials in measure variables. 
For instance, the running cost of \eqref{eq:J_intro} takes the form 
$\ell_i(\mu_t, X_t, \alpha_{i,t}) = \gamma [\mu]_{I_2} + (1-\gamma) |x|^2 + r_i |\alpha_{i,t}|^2 + \eta_i [\mu]_{e_i}^2$,
where we use the notation
\begin{equation}
    \label{eq:notation_intro}
    [\mu]_{Q} := \int_{\mathbb R^2} x^\top Q x \, \mu(dx), \hbox{ for } Q \in S_2(\mathbb R), 
    \quad
    [\mu]_{q} := \int_{\mathbb R^2} x^\top q \, \mu(dx), \hbox{ for } q \in \mathbb R^2,
\end{equation}
with $S_2(\mathbb R)$ denoting the set of $2 \times 2$ symmetric matrices, 
$I_2$ the identity matrix, and $e_i$ the $i$-th unit vector in $\mathbb R^2$.
Compared with existing approaches in the McKean-Vlasov control literature, 
such as solving 
the coupled forward-backward stochastic differential equations (FBSDEs) or using moment-based 
ansatz; see, e.g., \cite{CD18I, JSY24}, 
our direct approach to solving the Master equations 
\eqref{eq:master_intro1} and characterizing the Nash equilibrium through verification theorems 
is both novel and of independent interest.

We emphasize that this paper develops a general framework for ergodic nonzero-sum differential games 
with McKean-Vlasov dynamics; see Section \ref{sec:formulation} for detailed problem formulation. 
However, the specific model \eqref{eq:X_intro}--\eqref{eq:Nash_intro} and
the examples in Section \ref{sec:LQG} play a crucial role in 
developing the theoretical work in Section \ref{sec:general}, 
providing the intuition and illustrating the key concepts and techniques.
For instance, unlike the aforementioned non-uniqueness of $v_i$, 
a more subtle issue is the non-uniqueness of the constant $c_i$,
which is reflected by the explicit computations in Section \ref{sec:LQG}.
As another example of the insights gained from Section \ref{sec:LQG},
although the game's solution of \eqref{eq:X_intro}--\eqref{eq:Nash_intro} should theoretically 
be independent of the parameter $\gamma$ 
(since the cost functional $\hat J_i$ in \eqref{eq:J_intro} is invariant with respect to $\gamma$), 
our methodology requires solving a system of $\gamma$-dependent Master equations \eqref{eq:master_intro1}.
The explicit computations in Section \ref{sec:LQG} confirm that the resulting solution is indeed 
independent of $\gamma$, thereby validating our approach.
Additional structural insights into the game problem within 
our general framework are discussed as part of our concluding remarks and future research directions
in Section \ref{sec:conclusion}.

The paper is organized as follows.
Section \ref{sec:general} details the general problem formulation, the associated Master equations, and the verification theorem.
Section \ref{sec:LQG} applies this theory to Linear-Quadratic-Gaussian (LQG) settings, 
where the Nash equilibrium is shown to be determined by a system of algebraic Riccati equations.
Finally, Section \ref{sec:conclusion} offers concluding remarks and outlines future research directions.

In addition to the notations introduced in \eqref{eq:notation_intro},
we summarize other notations used throughout the paper.
Given a complete metric space $(\mathcal{S}, d)$ and some $p\ge 1$, 
we denote by $\mathcal{P}_p(\mathcal{S})$ the space of probability measures on 
$\mathcal{S}$ with finite $p$-th moment, equipped with the $p$-Wasserstein distance $\mathbb{W}_p$. 
For matrices $A$ and $B$, $|A|: = \sup_{|x|=1}  | Ax|$ denotes the operator norm of $A$, and $A\otimes B$ denotes the Kronecker product. 
For convenience, we write $A^{\otimes 2} = A\otimes A$. 
If $\mu\in \mathcal P_p(\mathcal S)$ and $f$ is a measurable function on $\mathcal S$, we 
often use the notation $\mu f := \int_{\mathcal S} f(x) \mu(dx)$ whenever the integral is well defined.

\section{Ergodic Non-zero Sum Differential Game in General Setting}
\label{sec:general}

\subsection{Problem Setup}
\label{sec:formulation}
Let $\left(\Omega, \mathcal{F}, \mathbb{P}, (\mathcal{F}_t)_{t \ge 0}\right)$ be a filtered probability space satisfying the usual conditions and 
equipped with a standard two-dimensional Brownian motion $\{W_t\}$.
Consider a controlled diffusion process $X_t\in \mathbb R^2$
governed by the following stochastic differential equation
\begin{equation}
    \label{eq:X}
        dX_t = b\left(\mu_t, X_t, \alpha_t \right) dt 
        + \sigma \left(\mu_t, X_t, \alpha_t \right) d W_t, 
\end{equation}
where
$\mu_{t} =   \mathcal L(X_t) \in \mathcal P_2(\mathbb R^2)$ is 
the law of $X_t$
for every $t\ge 0$, 
$\alpha_t \in \mathbb R^2$ is an $\mathcal F_t$-predictable control process,
and  the initial condition is $\mathcal L(X_0) = \mu_0 \in \mathcal P_2(\mathbb R^2)$.
Furthermore, we impose the following assumption on the  functions $b$ and $\sigma$:
\begin{assumption}
    \label{asm:bsigma}
The functions 
$b:  \mathcal P_2(\mathbb R^2) \times 
\mathbb R^2 \times \mathbb R^2
\to \mathbb R^2$ and 
$\sigma:  \mathcal P_2(\mathbb R^2) \times
\mathbb R^2 \times \mathbb R^2 \to \mathbb R^{2\times 2}$
are Lipschitz continuous, i.e., there exists
a constant $K>0$ such that
\begin{equation}
    \label{eq:b_lip} \begin{aligned}
    |b(\mu_1, x_1, a_1)  - b(\mu_2, x_2, a_2)|^{2} 
    & + |\sigma(\mu_1, x_1, a_1) - \sigma(\mu_2, x_2, a_2)|^{2}
  \\ & \le
    K(\mathbb W^{2}_2(\mu_1, \mu_2) + |x_1-x_2| ^{2}+ |a_1-a_2|^{2}),
\end{aligned}\end{equation}
for all 
$\mu_1, \mu_2 \in \mathcal P_2(\mathbb R^2)$ and $  x_1, x_2, a_1, a_2 \in \mathbb R^2$.
\end{assumption}

We consider a two-player differential game, where player $i$ ($i=1,2$) 
chooses the control $\alpha_{i,t}$ under the Nash equilibrium criterion.
The cost functional of agent $i$ $( = 1,2)$ is given by the long-term average 
of the instantaneous cost function 
$\ell_i: \mathcal P(\mathbb R^2) \times \mathbb R^2 \times \mathbb R \to \mathbb R$, namely 
\begin{equation}
    \label{eq:cost}
    \hat J_i(\alpha) = 
    \lim_{T\to \infty} \frac{1}{T} 
    \mathbb E \left[ \int_0^T 
    \ell_i( \mathcal L(X_t), X_t,  
    \alpha_{i,t}) dt
    \right].
\end{equation}
In the above, the instantaneous cost function 
$\ell_i$ depends on the control $\alpha$ only through its own control component $\alpha_i$.

To set up the Nash equilibrium, we define admissibility of the control process. 
\begin{definition}
    \label{def:admissible}
    We say a random process $\alpha = (\alpha_1, \alpha_2): [0,\infty) \times \Omega \to \mathbb R^2$ 
    is a stationary Markov control if there exists a measurable function $\bar{\alpha}: \mathcal P_2(\mathbb R^2) \times \mathbb R^2 \to \mathbb R^2$ 
    such that SDE \eqref{eq:X} admits a unique square integrable solution $X$ 
    under the closed-loop control $\alpha_t = \bar{\alpha}(\mathcal L(X_t), X_t)$ for all $t\ge 0$.
    Moreover, we say a stationary Markov control $\alpha$ is admissible if
    the associated process $(X_t, \alpha_t)$ converges in law to some distribution 
    $\theta_\infty^\alpha \in \mathcal P_2(\mathbb R^4)$, i.e., 
    \begin{equation}
        \label{eq:law-conv}
        \mathcal L(X_t, \alpha_t) \to \theta_\infty^\alpha \text{ in } \mathbb{W}_2(\mathbb{R}^4) 
        \text{ as } t \to \infty.
    \end{equation} Here and throughout this paper, the convergence of measures  is in the topology induced by the 2-Wasserstein distance.  
    We denote the set of all admissible controls by $\mathcal A$.
\end{definition}

\begin{assumption}
    \label{asm:control_space}
    The set of admissible controls $\mathcal A$ is nonempty.
\end{assumption}

We assume further that $\ell_i, i=1,2$ satisfy the following condition:
\begin{assumption}
    \label{asm:ell}
For each $i=1,2$, the function $\ell_i: \mathcal P_2(\mathbb R^2) \times \mathbb R^2 \times \mathbb R \to \mathbb R$ is locally Lipschitz continuous with at most quadratic growth, i.e., 
there exists a constant $K>0$ such that
\begin{equation}
    \label{eq:ell_lip}
    \begin{aligned}
    |&\ell_i  (\mu,  x, a) - \ell_i(\nu, y, b)|\\
    &  \le  K (1 + \mathbb W_2(\mu, \delta_{0}) + \mathbb W_2(\nu, \delta_0) + |x| + |y| + |a| + |b|) (\mathbb W_2(\mu, \nu) + |x - y| + |a - b|),
    \end{aligned}
\end{equation}
where $\delta_x$ is the Dirac measure at $x \in \mathbb R^2$.
\end{assumption}

Given an admissible control 
$\alpha = (\alpha_1, \alpha_2) \in \mathcal A$,
we also denote
$$\mathcal A_1(\bar \alpha) = 
\left\{
  \beta = (\beta_1, \beta_2) \in \mathcal A:
    \bar{\beta}_2 = \bar \alpha_2
\right\}, \quad
\mathcal A_2(\bar \alpha) =
\left\{
    \beta = (\beta_1, \beta_2) \in \mathcal A:
    \bar{\beta}_1 = \bar \alpha_1
\right\}.$$
The main objective of this paper is to find a pair of control functions
$\alpha^* = (\alpha_1^*, \alpha_2^*) \in \mathcal A$
such that
\begin{equation}
    \label{eq:NE}
    \begin{cases}
        \hat c_1 = 
        \hat J_1(\alpha^*) \le \hat J_1(\alpha), 
        & \forall \alpha \in \mathcal A_1(\alpha^*), 
        \\
        \hat c_2 = \hat J_2(\alpha^*)\le \hat J_2(\alpha),
        & \forall \alpha \in \mathcal A_2(\alpha^*).
    \end{cases}
\end{equation}
If both of the above inequalities hold, we call
the pair of controls 
$(\alpha_1^*, \alpha_2^*)$  a {\em  Nash equilibrium} of
the game, and the corresponding constants 
$(\hat c_1, \hat c_2)$ are called the {\em ergodic constants} 
of the game.

\begin{remark} \label{rem:assumptions}
 We finish this subsection with some remarks concerning the conditions \eqref{eq:b_lip}, \eqref{eq:law-conv}, and  \eqref{eq:ell_lip}. 
 First, \eqref{eq:b_lip} is a standard Lipschitz condition to 
 ensure the well-posedness of the McKean-Vlasov equation \eqref{eq:X} under 
 any admissible control $\alpha$, see e.g. \cite{CD18I}. 
 In fact, by the definition of admissible controls, 
 we have $ \mathbb E[\sup_{t\in [0, T]}|X_t|^2] < \infty$ for any $T> 0$. 
 Since this paper is  focused on long-term average problems, we further require that the 
 control $\alpha$ is such that the law of $(X_t, \alpha_t)$ 
 converges to some $\theta_\infty^\alpha \in \mathcal P_2(\R^4)$ in $\mathbb W_2$ as $t\to \infty$. Lemma \ref{lem:nonempty} below indicates that  such controls exist, i.e., the set $\mathcal A$ of admissible controls is nonempty under additional conditions.  
 
 Second, the running cost function $\ell_i$ is assumed to be locally Lipschitz continuous and satisfies a quadratic growth condition.  These conditions together with the convergence \eqref{eq:law-conv} ensures that the limit on the right hand side of \eqref{eq:cost} exists. Finally, we note that the local Lipschitz  and  quadratic growth condition  are quite common in the literature. The next example shows that these conditions are satisfied by a typical running cost function in LQG control problems.    
 \hfill$\square$
\end{remark} 
\begin{example}
    \label{ex:1}
    Consider the following running cost function in LQG control problems
    to be discussed in more details later in  Section \ref{sec:LQG}:
    $$\ell_i(\mu, x, a) = |x|^2 + r_i |a|^2 + [\mu]_{\theta_i}^2$$
    for some $r_i > 0$ and $\theta_i \in \mathbb R^2$, 
    where $[\mu]_{\theta_i}$ is defined as in \eqref{eq:notation_intro}.  
    The above $\ell_i$ satisfies \eqref{eq:ell_lip} thanks to  the following estimate:
    $$\big||x|^2 - |y|^2\big| \le (|x| + |y|) |x - y|, \ \text{ and } \ \big| |a|^2 - |b|^2\big| \le (|a| + |b|) |a - b|,$$ 
    and for any $q\in \R^2$, 
    \begin{displaymath}\big|[\mu]_q^2 - [\nu]_q^2\big| \le (|[\mu]_q| + |[\nu]_q|) |[\mu]_q - [\nu]_q| \le
    |q|_2^2 (\mathbb W_2(\mu, \delta_0) + \mathbb W_2(\nu, \delta_0)) \mathbb W_2(\mu, \nu).  \tag*{\qed}
   \end{displaymath}
\end{example}

\subsection{Flat Derivatives, Polynomials, and the Chain Rule}
\label{sec:flatderivative}

To proceed, we recall the definition of the flat derivative as introduced in \cite{Car19} (see also \cite[Chapter 2]{CDLL19} and \cite[Chapter 5]{CD18I} for its properties and the relationship with the $L$-derivative). Additionally, we introduce the concept of polynomials on the measure space $\mathcal P_2(\mathbb R^d)$, which plays a crucial role in our analysis of the Master equations \eqref{eq:master_intro1}.

\begin{definition}
    \label{def:flat}
    $($Definition 2.1 of \cite{Car19}$)$
    Given a functional $U: \mathcal P_2(\mathbb R^d)\mapsto \mathbb R$,
    we say that $U \in  \mathcal C^1$ if there exists a jointly continuous and
    bounded function $\frac{\delta U}{\delta \mu}: \mathcal P_2(\mathbb R^d)\times \mathbb R^d \mapsto \mathbb R$ such that 
    $$U(\mu') - U(\mu) = \int_0^1 
    \int_{\mathbb R^d} \frac{\delta U}{\delta \mu}( (1-h) \mu + h\mu', x) d(\mu' - \mu)(x) dh$$ 
    for all $\mu$ and $\mu'$ in $\mathcal P_2(\mathbb R^d)$.
    Moreover, we adopt the normalization convention
    \begin{equation*}
    \int_{\mathbb R^d}\frac{\delta U}{\delta \mu}(\mu, x) \mu(dx) = 0
    \end{equation*}
    for all $\mu\in \mathcal P_2(\mathbb R^d)$.
    
    Additionally, we say that $U$ is  partially-$\mathcal C^2$ 
    if $U\in \mathcal C^1$ and 
    $D_x \frac{\delta U}{\delta \mu}$ 
    and $D_{xx} \frac{\delta U}{\delta \mu}$  
    exist and are jointly continuous and bounded on
    $\mathcal P_2(\mathbb R^d)\times \mathbb R^d$.
    
    We say that $U$ is fully-$\mathcal C^2$ if 
    $\frac{\delta U}{\delta \mu}(\mu, x)$ 
    is $\mathcal C^1$ in 
    $\mu$ with a continuous and bounded derivative, namely 
    $$\frac{\delta^2 U}{\delta \mu^2}: \mathcal{P}_2(\mathbb R^d)\times \mathbb R^d \times \mathbb R^d 
    \mapsto \mathbb R$$ is continuous and bounded.
    \end{definition}

    It is often useful to use $L$-derivative
    for the analysis of the Hamilton-Jacobi-Bellman equations
    in the measure space (see \cite{CD18I}). 
    The connection of the flat derivative and the $L$-derivative 
    is given by Proposition 5.48 of \cite{CD18I}, that is,
    $$\partial_\mu U(\mu, x) = D_x \frac{\delta U}{\delta \mu}(\mu, x).$$

    Now, we are ready to introduce polynomial functions on the measure space $\mathcal P_2(\mathbb R^d)$ inductively.
    \begin{definition}
    \label{def:polynomial}
    A function $U(\mu): \mathcal P_2(\mathbb R^d) \times \mathbb R^d \mapsto \mathbb R$ is called a zero-order polynomial in $\mu$ 
    if it is a constant function in $\mu$.
    A function $U(\mu)$ is called an $n$-th order polynomial in $\mu$
    if its $L$-derivative $\partial_\mu U(\mu,x)$ is 
    a ($x$-dependent) $(n-1)$-th order polynomial in $\mu$.
    The collection of all $n$-th order polynomials is denoted by 
    $P_n$.
    Clearly, $P_n$ is a vector space for each $n\ge 0$ and any $U\in P_n$ 
    is fully-$\mathcal C^2$-differentiable.
    \end{definition}

    For example, with a given function $f\in C^1(\mathbb R^d, \mathbb R)$,
    the function $U(\mu) = \mu f$ is a first-order polynomial since 
    $$\partial_\mu U(\mu,x) = D_x f(x).$$ 
    Similarly, given $f, g \in C^1(\mathbb R^d, \mathbb R)$,
    the function $U(\mu) = (\mu f)(\mu g)$ is a second-order polynomial since 
    $$\partial_\mu U(\mu,x) = (\mu g) D_x f(x) + (\mu f) D_x g(x).$$
    Recall 
    $[\mu]_Q$ and $[\mu]_q$ defined in \eqref{eq:notation_intro} for any $Q \in S_2(\mathbb R)$ and $q \in \mathbb R^2$.
    Then, both $[\mu]_Q^m$ and $[\mu]_q^m$ are $m$-th order polynomials for any integer $m\ge 1$.

    To proceed, we 
    recall the chain rule for the McKean-Vlasov dynamics.
    \begin{lemma}
        \label{lem:chain}
      Consider the  stochastic differential equation  
      \eqref{eq:X}. 
        Suppose that the function $u: \mathcal P_2(\mathbb R^2) \to \mathbb R$ is 
        fully-$\mathcal C^2$-differentiable and satisfies  
        \begin{equation}
            \label{eq:chain_condition}
             \int_{\mathbb R^2}  \left|
             D_{xx}\frac{\delta u}{\delta \mu} 
                \left(\mu, x  \right)  
                \right|^2 d \mu(x)
             < \infty, \ \forall \mu \in \mathcal P_2(\mathbb R^2), 
        \end{equation}
        \begin{equation}
            \label{eq:chain_condition2}
            \mathbb E \left[ 
                \int_{0}^t\big( \left|
                    b
                    \left(\mu_s, X_s, 
                    \alpha_s
                    \right) \right|^2  
                    +
                    \left | \sigma\left(\mu_s, X_s, 
                \alpha_s 
                \right)\right|^4\big) ds
            \right] < \infty, 
        \end{equation}
    then
    we have
    \begin{equation}
        \label{eq:chain}
        \begin{aligned}
            u(\mu_t) &= u(\mu_0) + 
            \int_0^t \mathbb E \left[ 
                D_x \frac{\delta u}{\delta \mu}(\mu_s, X_s) 
                \cdot b(\mu_s, X_s, \alpha_s)
                \right] ds \\
            &\quad + \frac 1 2 \int_0^t \mathbb E \left[ 
                {\rm Tr} 
                \left\{D_{xx}\frac{\delta u}{\delta \mu} (\mu_s, X_s)
                \sigma \sigma^\top(\mu_s, X_s, \alpha_s) \right\} 
            \right] ds.
        \end{aligned}
    \end{equation}
    \end{lemma}
    \begin{proof}
        The proof is a direct application of 
        \cite[Theorem 5.92]{CD18I}.
    \end{proof}

\subsection{Master equations of HJB type}
\label{sec:HJB}

The first main result of this paper shows that 
    the ergodic constants of \eqref{eq:NE}
    can be characterized by the following 
    HJB equations: 
        \begin{equation}
        \label{eq:HJB}
        \begin{cases}
            \displaystyle
            \int_{\mathbb R^2} \inf_{a_1 \in \mathbb R} 
                \mathcal H_1 \left(\mu, x,  
                D_x \frac{\delta v_1}{\delta \mu}(\mu, x), D_{xx} \frac{\delta v_1}{\delta \mu}(\mu, x), 
                (a_1, \bar{\alpha}_2^*(\mu, x)) \right) \mu(d x)
             = c_1, \\
            \displaystyle
            \int_{\mathbb R^2} \inf_{a_2 \in \mathbb R}
                \mathcal H_2 \left(\mu, x,  
                D_x \frac{\delta v_2}{\delta \mu}(\mu, x), D_{xx} \frac{\delta v_2}{\delta \mu}(\mu, x), 
                (\bar{\alpha}_1^*(\mu, x), a_2) \right) \mu(d x)
             = c_2,\\
            \displaystyle
                \bar{\alpha}_1^*(\mu, x) = 
            \argmin_{a_1 \in \mathbb R} 
            \mathcal H_1 \left(
                \mu, x,  D_x \frac{\delta v_1}{\delta \mu}(\mu, x), D_{xx} \frac{\delta v_1}{\delta \mu}(\mu, x), 
            (a_1, \bar{\alpha}_2^*(\mu, x)) \right), \\
            \displaystyle
                \bar{\alpha}_2^*(\mu, x) = \argmin_{a_2 \in \mathbb R} 
            \mathcal H_2 \left(\mu, x,  D_x \frac{\delta v_2}{\delta \mu}(\mu, x), D_{xx} \frac{\delta v_2}{\delta \mu}(\mu, x), 
            (\bar{\alpha}_1^*(\mu, x), a_2) \right), 
        \end{cases}
    \end{equation} 
    for all $(\mu, x) \in  \mathcal P_2(\mathbb R^2) \times \mathbb R^2$, 
    where $\mathcal H_i:  \mathcal P_2(\mathbb R^2) \times \mathbb R^2 \times \mathbb R^2 \times S_2(\mathbb R) \times \mathbb R^2 \to \mathbb R$ 
    is   the Hamiltonian function given by
    \begin{equation}
        \label{eq:calH}
        \mathcal H_i(\mu, x,  p, Q, a) = 
        p \cdot b(\mu, x,  a) + 
        \frac 1 2 \text{Tr}(Q \sigma \sigma^\top (\mu, x, a)) 
        + \ell_i(\mu, x,  a_i),
    \end{equation} for any $(\mu, x,  p, Q, a) \in  \mathcal P_2(\mathbb R^2) \times \mathbb R^2 \times \mathbb R^2 \times S_2(\mathbb R^{2}) \times \mathbb R^2$.

    \begin{definition} \label{def:1}
        We say that the six-tuple $(v_1, v_2, c_1, c_2, \bar{\alpha}_1^*, \bar{\alpha}_2^*)$ 
        is a solution of \eqref{eq:HJB} if 
        \begin{itemize}
            \item[(i)] $v_1, v_2: \mathcal P_2(\mathbb R^2) \mapsto \mathbb R^2$ are fully-$\mathcal C^2$ functions,
            \item[(ii)] $c_1, c_2 \in \mathbb R$ are constants,  
            \item[(iii)]   $(\bar{\alpha}_1^*, \bar{\alpha}_2^*)^\top: \mathcal P_2(\mathbb R^2) \times \mathbb R^2 \mapsto \mathbb R^2$ 
            is a Borel measurable function, and 
            \item[(iv)]    \eqref{eq:HJB} is 
        satisfied for all 
        $(\mu, x) \in  \mathcal P_2(\mathbb R^2) \times \mathbb R^2$.
        \end{itemize} 
              Note that, if 
        $(v_1, v_2, c_1, c_2, \bar{\alpha}_1^*, \bar{\alpha}_2^*)$ 
        is a solution of \eqref{eq:HJB}, so is 
        $(v_1 + K_1, v_2 + K_2, c_1, c_2, \bar{\alpha}_1^*, \bar{\alpha}_2^*)$
         for any constants $K_1, K_2\in \R$.
        Therefore, the uniqueness of the solution 
        of \eqref{eq:HJB}
        does not hold for an obvious reason.
    \end{definition}

We now present the following verification theorem 
for the Nash equilibrium $\bar{\alpha}^*$ and 
the ergodic constants $\hat c_i$ defined in the ergodic Nash 
game \eqref{eq:NE}.                 

\begin{proposition}
    \label{p:vt}
    Let Assumptions \ref{asm:bsigma}, \ref{asm:control_space}, and \ref{asm:ell} hold.
    Suppose that $(v_1, v_2, c_1, c_2, \bar{\alpha}_1^*, \bar{\alpha}_2^*)$ is a solution 
    to the HJB equations \eqref{eq:HJB} such that the Markov control $\alpha^*$ 
    associated with the feedback form $\bar{\alpha}^*$, i.e., $\alpha_t^* = \bar{\alpha}^*(\mu_t, X_t)$
    for all $t\ge 0$, is admissible with corresponding state process
    $X^*$. Moreover, assume that
    the functions $v_1, v_2$ satisfy the conditions \eqref{eq:chain_condition} 
    and \eqref{eq:chain_condition2}.
    Then, $(\alpha_1^*, \alpha_2^*)$ is a Nash equilibrium for the game \eqref{eq:NE}, 
    and the corresponding ergodic constants are given by $\hat c_i = c_i$ for $i=1,2$.
\end{proposition}

\begin{proof} 
Let $\alpha^*$ be as in the statement given by the feedback form $\ba^*$.
Hence, for an arbitrary admissible control $\alpha \in \mathcal A_1(\alpha^*)$,
applying Lemma \ref{lem:chain}  with the function $v_1$, and using the Hamiltonian $\mathcal H_i$ defined in \eqref{eq:calH}, 
we obtain with ${\bf X}_s = (\mu_s, X_s)$ that
\begin{equation}\label{eq:vt_pf1}
    \begin{aligned}
        &\mathbb E \left[ v_1(\mu_t) + 
        \int_0^t \ell_1(\mu_s, X_s,  
        \alpha_{1,s})  ds \right]  \\
        &\hspace{.3in} = v_1(\mu_0) + 
        \int_0^t \mathbb E \left[ 
            \mathcal H_1 \left({\bf X}_s,  D_x \frac{\delta v_1}{\delta \mu}({\bf X}_s), D_{xx} \frac{\delta v_1}{\delta \mu}({\bf X}_s), (\alpha_{1,s}, \bar{\alpha}^*_2({\bf X}_s)) \right) \right] ds.
    \end{aligned}
\end{equation}
By the first and the third equations of 
\eqref{eq:HJB}, we have 
$$\mathbb E \left[
    \mathcal H_1 \left({\bf X}_s,  D_x \frac{\delta v_1}{\delta \mu}({\bf X}_s), D_{xx} \frac{\delta v_1}{\delta \mu}({\bf X}_s), (\alpha_{1,s}, \bar{\alpha}^*_2({\bf X}_s)) \right) 
    \right] \ge c_1,$$
for all $s\ge 0$ and the equality holds if 
$(X, \alpha) = (X^*, \alpha^*)$,
where $X^*$ is the state 
process satisfying \eqref{eq:X} corresponding to 
the control  $\alpha^*$.
Therefore, we can rewrite \eqref{eq:vt_pf1} as 
\begin{equation}
    \label{eq:vt_pf2}
    v_1(\mu_t) + \mathbb E \left[  \int_0^t 
    \left( \ell_1(\mu_s, X_s, \alpha_{1,s}) - c_1 \right) ds \right] 
    \ge v_1(\mu_0),
\end{equation}
and the equality holds for
$(X, \alpha) = (X^*, \alpha^*)$.

Dividing both sides of \eqref{eq:vt_pf2} by $t$ and then passing to the limit as $t\to\infty$, 
we obtain the following inequality
\begin{equation}
    \label{eq1:vt_pf3}
    \lim_{t\to \infty} \frac 1 t 
    \mathbb E \left[  \int_0^t \left( \ell_1(\mu_s, X_s, \alpha_{1,s}) - c_1 \right) ds \right] \ge 0,
\end{equation}
and the equality holds for
$(X, \alpha) = (X^*, \alpha^*)$ which 
inherits  from \eqref{eq:vt_pf2}. 
As we observed in Remark \ref{rem:assumptions}, the limit on the 
left-hand side of \eqref{eq1:vt_pf3} exists for any admissible control 
$\alpha \in \mathcal A$.
Therefore, we have 
$$c_1 \le \lim_{t\to \infty} \frac 1 t 
\mathbb E \left[  \int_0^t \ell_1(\mu_s, X_s, \alpha_{1,s}) ds \right] 
= \hat J_1(\alpha),$$
for all $\alpha = (\alpha_1, \alpha_2^*) \in \mathcal A_1(\alpha^*)$.
This implies that
\begin{equation}
    \label{eq1:vt_pf4}
    c_1 = \hat J_1(\alpha^*) = 
    \inf_{\alpha \in \mathcal A_1(\alpha^*)} 
    \hat J_1(\alpha). 
\end{equation}
Similarly, we can prove that
\begin{equation}
    \label{eq:vt_pf5}
    c_2 = \hat J_2(\alpha^*) = 
    \inf_{\alpha \in \mathcal A_2(\alpha^*)} \hat J_2(\alpha). 
\end{equation}
From \eqref{eq1:vt_pf4} and \eqref{eq:vt_pf5},
we have $\hat c_i = c_i = \hat J_i(\alpha^*)$
for $i = 1, 2$.
\end{proof}

\begin{remark}
    The admissibility requirement for the control $\alpha^*$, 
    specifically the convergence of the law $\mathcal{L}(X_t^*, \alpha_t^*)$ in $\mathbb{W}_2$ as $t \to \infty$,
    plays a critical role in identifying the ergodic constants $\hat c_i$ with the constants $c_i$ appearing in \eqref{eq:HJB}.
    This condition is essential because the constants $c_i$ are not uniquely determined by the HJB equations \eqref{eq:HJB} alone.
    We refer the reader to \eqref{eq:c1c2} and the rest of Section \ref{sec:LQG} for an explicit example demonstrating the non-uniqueness of $c_i$. \hfill$\square$
\end{remark}

Proposition \ref{p:vt} gives a partial verification in that 
it does not provide a clear implication of the 
functions
$(v_1, v_2)$ from the solution of the HJB equations \eqref{eq:HJB}.
The lack of this verification of $(v_1, v_2)$ 
is due to the non-uniqueness of the solution of the HJB equations \eqref{eq:HJB}.
In the following, we establish the verification of $(v_1, v_2)$
by showing that they are a constant shift
of the value function of an auxiliary control problem under some additional conditions.

\subsection{The complete verification: Auxiliary control problem}
\label{sec:verification}
We have shown  in Proposition \ref{p:vt} that
 if  $(v_1, v_2, c_1, c_2, \bar{\alpha}_1^*, \bar{\alpha}_2^*)$ is a solution  
 of HJB equations \eqref{eq:HJB}, 
 then the feedback control $\alpha^*$ generated by $\bar{\alpha}^* = (\bar{\alpha}_1^*, \bar{\alpha}_2^*)$ 
provides a Nash equilibrium of \eqref{eq:NE}, and $(c_1, c_2) $ are the corresponding ergodic constants.
A natural question is:  
 What are the functions $v_1$ and $v_2$ in the context of the    
   control problem? 
To answer this question, we consider the following auxiliary control problem.

Recall that $\hat c_i, i=1,2$ are the ergodic constants associated with 
Nash equilibrium controls defined in \eqref{eq:NE}.
Consider, for any $(\mu_0, \alpha)$ in $\mathcal{P}(\mathbb{R}^2) \times \mathcal{A}$
\begin{equation} \label{eq:J}
    J_i(\mu_0, \alpha) = \lim_{T \to \infty} \mathbb{E} 
    \left[ \int_0^T \left( \ell_i(\mathcal L(X_t), X_t, \alpha_{i,t}) - \hat c_i \right) dt \right], \quad i = 1, 2.
\end{equation}
The objective of this problem is to find a 
Nash equilibrium $\alpha^* \in \mathcal A$ satisfying
\begin{equation}
    \label{eq:NE2}
    \begin{cases}
        V_1(\mu_0) := J_1(\mu_0, \alpha^*) 
        \le J_1(\mu_0, \alpha), 
        & \forall \alpha \in 
        \mathcal A_1(\alpha^*), \medskip\\
        V_2(\mu_0) := J_2(\mu_0, \alpha^*) 
        \le J_2(\mu_0, \alpha), 
        & \forall \alpha \in \mathcal A_2(\alpha^*).
    \end{cases}
\end{equation} 

To facilitate later presentation, we define the  projection maps:
$$\pi^x(x_1, x_2, a_1, a_2) = (x_1, x_2), \
\pi^i(x_1, x_2, a_1, a_2) = a_i, \ i =1, 2, \ \  \forall (x_1, x_2, a_1, a_2) \in \mathbb R^4,$$ 
as well as the extension of the running cost 
$\ell_i: \mathcal P_2(\mathbb R^2) \times \mathbb R^2 \times \mathbb R \to \mathbb R$ 
to $\hat \ell_i: \mathcal P_2(\mathbb R^4) \times \mathbb R^4 \to \mathbb R$ defined
by
$$\hat \ell_i(\theta, x, a) = \ell_i(\pi^x_\# \theta, x, a_i), \ \ \forall (\theta, x, a) \in \mathcal P_2(\mathbb R^4) \times \mathbb R^4, \ i =1, 2,$$
where $\pi^x_\# \theta$ is the push-forward measure of $\theta$ by the projection map $\pi^x$. 
Then, we can rewrite the objective functions as a function of 
the joint distribution of the state and the control, that is,
\begin{equation} \label{eq:ell_bar}
    \bar \ell_i(\theta) = 
    \int_{\mathbb R^4} 
    \hat \ell_i(\theta, x, a) \theta(dx, da),\ \ i =1,2, 
\end{equation}
Note that for an $\mathbb R^2 \times \mathbb R \times \mathbb R$-valued random variable
$(X, A_1, A_2)$ 
with law $\theta \in \mathcal P_2(\mathbb R^4)$,
we can write $$ \mathbb E[\ell_i(\mathcal L(X), X, A_i)] = \bar \ell_i(\theta), \ \ i =1, 2.$$

In particular, with the above notation, and noting that 
$\E[\ell_i(\mathcal L(X_t), X_t, \alpha_t)] 
\to \bar\ell_i(\theta_\infty^\alpha)$, $i=1,2$ 
as $t\to\infty$ for any $\alpha \in \mathcal A$, where $\theta_\infty^{\alpha}$ 
is the limit distribution of 
$(X_t, \alpha_t)$ as in \eqref{eq:law-conv}, we can rewrite the equilibrium condition \eqref{eq:NE} 
as follows:
\begin{equation}
    \label{eq:vc1}
    \begin{cases}
        \hat c_1 = \bar \ell_1(\theta_\infty^*) 
        = \inf_{\alpha \in 
        \mathcal A_1(\alpha^*)} \bar \ell_1(\theta_\infty^{\alpha}), \medskip\\
        \hat c_2 = \bar \ell_2(\theta_\infty^*) = 
        \inf_{\alpha \in \mathcal A_2(\alpha^*)} \bar \ell_2(\theta_\infty^{\alpha}),
    \end{cases}
\end{equation}
 where 
$\theta_\infty^*$ is the limit distribution 
of $(X^*_t, \alpha^*_t)$ 
corresponding to $\alpha^* \in \mathcal A$. 
Note that $\bar\ell_i(\theta_\infty^\alpha)$ and  $ \bar \ell_i(\theta_\infty^*)$,
for $i=1,2$, are finite since $\ell_i$ has quadratic growth and $\theta_\infty^\alpha, \theta_\infty^* \in \cP_2(\R^4)$.

Recall that $\ell_i$ is assumed to be locally Lipschitz continuous with quadratic growth
in Assumption \ref{asm:ell}. Indeed, we can show that $\bar \ell_i$ introduced above is also locally Lipschitz continuous under the same condition.

\begin{assumption}
    \label{asm:vi}
    The solution $(v_1, v_2, c_1, c_2, \bar{\alpha}_1^*, \bar{\alpha}_2^*)$ of the HJB equations \eqref{eq:HJB}
    satisfies the following conditions:
    $D_x \frac{\delta v_i}{\delta \mu}$ is Lipschitz continuous and 
            $D_{xx} \frac{\delta v_i}{\delta \mu}$ is continuous and bounded. 
\end{assumption}
\begin{lemma}
    \label{lem:ell} 
    \begin{enumerate}
        \item[(i)] For a function $f: \mathcal P_2(\mathbb R^4) \times \mathbb R^4 \to \mathbb R$,
        we define 
        $$\hat f(\theta) = \int_{\mathbb R^4} f(\theta, y) \theta(dy).$$
        If $f$ is locally Lipschitz continuous and has quadratic growth, i.e., 
        $$|f(\theta, y) - f(\gamma, z)| \le K (1 + \mathbb W_2(\theta, \delta_0) 
        + \mathbb W_2(\gamma, \delta_0) + |y| + |z|) (\mathbb W_2(\theta, \gamma) + |y-z|),$$
        for some constant $K > 0$ and for all $\theta, \gamma \in \mathcal P_2(\mathbb R^4)$ and $y, z \in \mathbb R^4$. 
        Then, $\hat f$ is locally Lipschitz continuous with respect to the 2-Wasserstein distance. 
        In particular, 
        if $\theta_t \to \theta_\infty$ in $\mathbb W_2$, then $\hat f(\theta_t) \to \hat f(\theta_\infty)$ as $t\to \infty$.
        \item[(ii)] In particular, under Assumption \ref{asm:ell}, $\bar \ell_i$ is locally Lipschitz continuous with respect to the 2-Wasserstein distance for $i=1,2$.
         
        \item[(iii)] For any $\alpha \in \mathcal A$, define 
        $\bar{\mathcal H}_{i,t}^\alpha$ as
            \begin{align*}
                \bar{\mathcal H}_{i,t}^\alpha & = \mathbb E \bigg[
                    \mathcal H_i \bigg(\mu_t, X_t,  
                     D_x \frac{\delta v_i}{\delta \mu}(\mu_t, X_t), D_{xx} \frac{\delta v_i}{\delta \mu}(\mu_t, X_t), 
                    \alpha_t \bigg) \bigg],
            \end{align*} where $\mathcal H_i$ is the Hamiltonian function defined in \eqref{eq:calH}, $X_t$ is the controlled state process under $\alpha\in \mathcal A$, and $\mu_t$ is the distribution of $X_t$.
            If Assumptions \ref{asm:bsigma}, \ref{asm:ell}, and  \ref{asm:vi} hold, 
            then $\bar{\mathcal H}_{i,t}^\alpha$ is convergent as $t\to \infty$ 
            to some constant $\bar{\mathcal H}_{i, \infty}^\alpha$.
    \end{enumerate}
\end{lemma}
\begin{proof}
    \begin{enumerate}
        \item[(i)] For any  $R > 0$, 
     define the ball 
     $B_R^{\mathbb W_2} = \{\theta \in \mathcal P_2(\mathbb R^4): \mathbb W_2(\theta, \delta_{\mathbf 0}) < R\}$, 
    where $\delta_{\mathbf 0}$ is the Dirac measure at $\mathbf 0 \in \mathbb R^4$.
    For any $\theta, \gamma \in B_R^{\mathbb W_2}$, 
    we will show that 
    \begin{equation}
        \label{eq:ell_lip1}
        |\hat f(\theta) - \hat f(\gamma)| \le C_R \mathbb W_2(\theta, \gamma),
    \end{equation}
    for some constant $C_R > 0$ depending only  on $R$.

    We denote $d = \mathbb W_2(\theta, \gamma)$.
    By the definition of Wasserstein distance,
    there exist random variables 
    $Y_n\sim \theta$ and $Z_n \sim \gamma$ such that 
    $$d^2 = \lim_{n\to \infty} \mathbb E[|Y_n - Z_n|^2].$$
    Then, we have 
    $$\hat f(\theta) = \mathbb E[f(\theta, Y_n)], \ \ \hat f(\gamma) = \mathbb E[f(\gamma, Z_n)], 
    \ \ \forall n\ge 1.$$
    Hence, we can write
    \begin{align*}
        |\hat f(\theta) - \hat f(\gamma)| 
        & = \lim_{n\to \infty} 
        |\mathbb E[f(\theta, Y_n) - f(\gamma, Z_n)]| \\
        & \le \lim_{n\to \infty} 
        \mathbb E[|f(\theta, Y_n) - f(\gamma, Z_n)|].
    \end{align*}
    Using the conditions that $f$ is locally Lipschitz continuous and has quadratic growth, we have 
    \begin{align*}
        |\hat f(\theta) & - \hat f(\gamma)| \le  \lim_{n\to \infty} 
        \mathbb E[|f(\theta, Y_n) - f(\gamma, Z_n)|] \\
        & \le K \lim_{n\to \infty} 
        \mathbb E[(1 + \mathbb W_2(\theta, \delta_0) +
        \mathbb W_2(\gamma, \delta_0) + |Y_n| + |Z_n|) \cdot (\mathbb W_2(\theta, \gamma) + |Y_n - Z_n|)] \\
         & \le K \lim_{n\to \infty} 
        \mathbb E[(1 +2R + |Y_n| + |Z_n|) 
         (d + |Y_n - Z_n|)] \\
         & \le K \lim_{n\to \infty} 
        \left(\mathbb E[(1 +2R + |Y_n| + |Z_n|)^2] \right)^{1/2} 
          \cdot\lim_{n\to \infty} \left( \mathbb E[(d + |Y_n - Z_n|)^2] \right)^{1/2}, 
    \end{align*}
    where we have used the Cauchy-Schwarz inequality in the last step. Since $\theta, \gamma \in B_R^{\mathbb W_2}$, we have
    $$\lim_{n\to \infty} \left(\mathbb E[(1 +2R + |Y_n| + |Z_n|)^2] \right)^{1/2} \le K R $$
    for some constant $K > 0$ independent of $R$. Also, we have
    $$\lim_{n\to \infty} \left( \mathbb E[(d + |Y_n - Z_n|)^2] \right)^{1/2} \le K d$$
    due to $d = \mathbb W_2(\theta, \gamma)$. This proves \eqref{eq:ell_lip1}.
    At last, for any $\theta_t \to \theta$ in $\mathbb W_2$, continuity of $\hat f$ implies $\hat f(\theta_t) \to \hat f(\theta)$ as $t\to \infty$.
    \item[(ii)] Note that, if we denote $y = (x, a) \in \mathbb R^4$, then $\hat \ell_i(\theta, y) = \ell_i(\pi^x_\# \theta, x, \pi^i(y))$ 
    for $i=1,2$. Since $\ell_i$ is locally Lipschitz continuous and has quadratic growth, 
    and $\theta \mapsto \pi^x_\# \theta$  and $y \mapsto \pi^i(y)$ are Lipschitz continuous, we
    can conclude that $\hat \ell_i$ is locally Lipschitz continuous and has quadratic growth. Then, by the first part of the lemma, we have $\bar \ell_i$ is locally Lipschitz continuous for $i=1,2$. 
    Applying the first part of the lemma to $\hat \ell_i$ with $i=1,2$ respectively, we have 
    the Lipschitz continuity of $\bar \ell_i$ for $i=1,2$.

    \item[(iii)] Note that $\bar{\mathcal H}_{i,t}^\alpha$ can be rewritten as
    $$\bar{\mathcal H}_{i,t}^\alpha = \mathbb E \left[ \hat{\mathcal H}_i(\theta_t, X_t, \alpha_t) \right]$$
    for some function $\hat{\mathcal H}_i$ satisfying
    \begin{align*}
    \hat{\mathcal H}_i(\theta_t, X_t, \alpha_t) &= 
    D_x \frac{\delta v_i}{\delta \mu}(\mu_t, X_t) \cdot b(\mu_t, X_t,  \alpha_{i,t}) 
        \\ & \quad
       + \frac 1 2 \text{Tr}(D_{xx} \frac{\delta v_i}{\delta \mu}(\mu_t, X_t) \sigma \sigma^\top (\mu_t, X_t, \alpha_{i,t})) 
        + \ell_i(\mu_t, X_t,  \alpha_{i,t}),
    \end{align*} where $\theta_t$ is the distribution of $(X_t, \alpha_t)$. Note that $\pi^x_\# \theta_t = \mu_t$.
   Thanks to Assumptions \ref{asm:bsigma}, \ref{asm:ell}, and \ref{asm:vi},  we can conclude that
    $\hat{\mathcal H}_i$ is locally Lipschitz continuous and satisfies a quadratic growth condition. 
    Then, using the similar argument as in the second part of the proof, 
    together with the convergence of $\theta_t$ to $\theta_\infty^\alpha$ 
    for any $\alpha \in \mathcal A$,
    we can show that $\bar{\mathcal H}_{i,t}^\alpha$ is convergent as $t\to \infty$ to some constant $\bar{\mathcal H}_{i, \infty}^\alpha$. \qedhere
    \end{enumerate}
    \end{proof}

We are now ready to present the full version of the verification theorem
as an extended result of Proposition \ref{p:vt}.
Note that Proposition \ref{p:vt} shows the verification theorem
for the Nash equilibrium generated by $(\alpha_1^*, \alpha_2^*)$ 
and the ergodic constants $(\hat c_1, \hat c_2)$. 
Therefore it remains to show the verification theorem 
for the value functions $v_i$ of the HJB equations \eqref{eq:HJB}
to the auxiliary control problem. 
Different from the classical 
verification theorem, one can not simply 
equate $V_i$ from the game problem and the solution $v_i$ of 
HJB equation due to the nonuniqueness of $v_i$.

We need the following assumption: 
\begin{assumption}
    \label{asm:calH}
    The mapping $a_i \mapsto \mathcal H_i(\mu, x, p, Q, a)$ is strictly convex, where the Hamiltonian $\mathcal H_i$ is defined in \eqref{eq:calH}.
\end{assumption}

\begin{theorem}
    \label{thm:vt1}
   Let   Assumptions \ref{asm:bsigma} - \ref{asm:calH} hold.
    Suppose that $(v_1, v_2, c_1, c_2, \bar{\alpha}_1^*, \bar{\alpha}_2^*)$ is a solution 
    to the HJB equations \eqref{eq:HJB} such that 
     the control $\alpha^*$ generated by
    $\bar{\alpha}^*$ is admissible with corresponding state process $X^*$, 
    and the functions $v_1, v_2$ satisfy the conditions \eqref{eq:chain_condition} 
    and \eqref{eq:chain_condition2}.
    If the limit distribution $\mu^*_\infty =  \lim_{t\to \infty}\mathcal L(X_t^*)$ 
    is the unique invariant measure with respect to the control $\alpha^*$,
    then we have 
    \begin{description}
     \item[(i)] $\alpha^*$ is a Nash equilibrium for the game \eqref{eq:NE}, 
     \item[(ii)]  the corresponding ergodic constants are given by $\hat c_i = c_i$ for $i=1,2$,  and
     \item[(iii)]   the value functions $V_i$ of the auxiliary control problem \eqref{eq:NE2} are given by
    $$V_i(\mu_0) = v_i(\mu_0) - v_i(\mu_\infty^*), \quad i=1,2.$$ 
    \end{description}
\end{theorem}
\begin{proof}
In view of Proposition \ref{p:vt}, we only need to show the identity 
given in (iii) about the 
value function $V_i$ of the auxiliary control problem \eqref{eq:NE2}.

For any $\alpha \in \mathcal A_1(\alpha^*)$,
passing to the limit as $ {t\to \infty}$ on both 
sides of \eqref{eq:vt_pf1}, we have
\begin{equation}\label{eq:cvt_pf0}
    J_1(\mu_0, \alpha) = v_1(\mu_0) - v_1(\mu_\infty^{\alpha}) 
    + \lim_{t\to\infty} \int_0^t (\bar{\mathcal H}_{1,s}^\alpha - \hat c_1) ds,
\end{equation}
where 
$$\bar{\mathcal H}_{1,s}^\alpha = \mathbb E \left[
    \mathcal H_1 \left(\mu_s, X_s,  D_x \frac{\delta v_1}{\delta \mu}(\mu_s, X_s), 
    D_{xx} \frac{\delta v_1}{\delta \mu}(\mu_s, X_s), 
    (\alpha_{1,s}, \bar{\alpha}^*_2(\mu_s, X_s)) \right) 
    \right].$$
Since $v_1$ is the solution of the HJB equations \eqref{eq:HJB}, we
have $\bar{\mathcal H}_{1,s}^\alpha \ge \hat c_1$ for all $s\ge 0$ and 
\begin{equation}
    \label{eq:cvt_pf1}
    J_1(\mu_0, \alpha) \ge  
    v_1(\mu_0) - v_1(\mu_\infty^{\alpha}),
\end{equation} 
where  $\mu_t^\alpha: = \mathcal L(X_t)$, and
$\mu_t^\alpha$ converges to some distribution denoted by 
$\mu^\alpha_\infty \in \mathcal P_2(\R^2)$
as $t\to\infty$.
Moreover, if we take $\alpha = \alpha^*$ in \eqref{eq:cvt_pf1}, we have
\begin{equation}
    \label{eq:cvt_pf2}
    J_1(\mu_0, \alpha^*) = v_1(\mu_0) - v_1(\mu_\infty^*), 
\end{equation}
which implies that 
\begin{equation}
    \label{eq:cvt_pf3}
    V_1(\mu_0) < \infty.
\end{equation}
We rewrite $\bar{\mathcal H}_{1,s}^\alpha$ as
\begin{align*}
    \bar{\mathcal H}_{1,t}^\alpha & = \mathbb E \Big[
        D_x \frac{\delta v_1}{\delta \mu}(\mu_t, X_t) \cdot b(\mu_t, X_t,  \alpha_{1,t}) +
        \\ & \quad
        \frac 1 2 \text{Tr}(D_{xx} \frac{\delta v_1}{\delta \mu}(\mu_t, X_t) \sigma \sigma^\top (\mu_t, X_t, \alpha_{1,t})) 
        + \ell_1(\mu_t, X_t,  \alpha_{1,t}),
    \Big] 
\end{align*}
Since $(X_t, \alpha_t)$ converges in distribution to $\theta_\infty^\alpha$ as
$t\to\infty$, there exists a limit of $\bar{\mathcal H}_{1,t}^\alpha$ as 
$t\to\infty$, which we denote by $\bar{\mathcal H}_{1,\infty}^\alpha$
due to Lemma \ref{lem:ell} (iii).

If $\alpha$ satisfies $\bar{\mathcal H}_{1,\infty}^\alpha > \hat c_1$, we have $J_1(\mu_0, \alpha) = \infty$ 
and hence $\alpha$ is not optimal due to \eqref{eq:cvt_pf3}.
In other words, we only need to consider $\alpha \in \mathcal A_1(\alpha^*)$ such that 
$\bar{\mathcal H}_{1,\infty}^\alpha = \hat c_1$. 
Note that
\begin{align*}
    \bar{\mathcal H}&_{1,\infty}^\alpha  = 
\int_{\mathbb R^2} \mathcal H_1
\left(
    \mu_\infty^{\alpha}, x, D_x \frac{\delta v_1}{\delta \mu}(\mu_\infty^{\alpha}, x), D_{xx} \frac{\delta v_1}{\delta \mu}(\mu_\infty^{\alpha}, x), 
    (\bar{\alpha}_{1}(\mu_\infty^{\alpha}, x), \bar{\alpha}^*_2(\mu_\infty^{\alpha}, x)) \right) \mu_\infty^{\alpha}(dx).\end{align*}
As $v_1$ is a solution to the HJB equation \eqref{eq:HJB}, we have
$\bar{\mathcal H}_1(\mu_\infty^{\alpha}, {\alpha}_{1,\infty}) \ge \hat c_1$. 
Set $A: = \{x \in \R^2: \bar{\alpha}_{1}(\mu_\infty^{\alpha}, x) \neq \bar{\alpha}_1^*(\mu_\infty^{\alpha}, x)\}$. Note that \begin{align*}
\mathcal H_1 &  \left(\mu_\infty^\alpha, x,  D_x \frac{\delta v_1}{\delta \mu}(\mu_\infty^\alpha, x), 
D_{xx} \frac{\delta v_1}{\delta \mu}(\mu_\infty^\alpha, x),  (\bar\alpha_{1}(\mu_\infty^\alpha, x), \bar{\alpha}^*_2(\mu_\infty^\alpha, x)) \right)\\ &  \ge \mathcal H_1 \left(\mu_\infty^\alpha, x,  D_x \frac{\delta v_1}{\delta \mu}( \mu_\infty^\alpha, x), 
D_{xx} \frac{\delta v_1}{\delta \mu}( \mu_\infty^\alpha, x), (\ba_1^*(\mu_\infty^\alpha, x), \bar{\alpha}^*_2(\mu_\infty^\alpha, x)) \right), \quad \forall x \in \R^2;
\end{align*} the above inequality is a strict inequality on $A$ thanks to Assumption \ref{asm:calH}. If $\mu_\infty^\alpha(A) > 0$,   then by  \eqref{eq:HJB}, \begin{align*}
\int_{\R^2}&  \mathcal H_1   \left(\mu_\infty^\alpha, x,  D_x \frac{\delta v_1}{\delta \mu}(\mu_\infty^\alpha, x), 
D_{xx} \frac{\delta v_1}{\delta \mu}(\mu_\infty^\alpha, x),  (\bar\alpha_{1}(\mu_\infty^\alpha, x), \bar{\alpha}^*_2(\mu_\infty^\alpha, x)) \right)\mu_\infty^{\alpha}(dx) \\ 
& >\int_{A}  \mathcal H_1 \left(\mu_\infty^\alpha, x,  D_x \frac{\delta v_1}{\delta \mu}( \mu_\infty^\alpha, x), 
D_{xx} \frac{\delta v_1}{\delta \mu}( \mu_\infty^\alpha, x), (\ba_1^*(\mu_\infty^\alpha, x), \bar{\alpha}^*_2(\mu_\infty^\alpha, x)) \right)\mu_\infty^{\alpha}(dx)  \\ 
&\ \  + \int_{\R^2\setminus A}  \mathcal H_1 \left(\mu_\infty^\alpha, x,  D_x \frac{\delta v_1}{\delta \mu}( \mu_\infty^\alpha, x), 
D_{xx} \frac{\delta v_1}{\delta \mu}( \mu_\infty^\alpha, x), (\ba_1^*(\mu_\infty^\alpha, x), \bar{\alpha}^*_2(\mu_\infty^\alpha, x)) \right)\mu_\infty^{\alpha}(dx)\\
& = \hat c_1.
\end{align*} This is a contradiction to the assumption that $\bar{\mathcal H}_{1,\infty}^\alpha = \hat c_1$. 
Hence, we must have $\mu_\infty^\alpha(A) = 0$, which, in turn,  implies that  under the assumption that $\bar{\mathcal H}_{1,\infty}^\alpha = \hat c_1$, we necessarily have 
$$\bar{\alpha}_{1}(\mu_\infty^{\alpha}, x) = \bar{\alpha}_1^*(\mu_\infty^{\alpha}, x),
\ \ \mu_\infty^{\alpha}\text{-a.e. } x \in \mathbb R^2.$$

Therefore, we have $\alpha = \alpha^*$ and $\mu_\infty^{\alpha} = \mu_\infty^*$ due to 
the uniqueness of the invariant measure $\mu_\infty^*$ in the theorem statement.
Summarizing the above arguments, we have shown that, 
$$V_1(\mu_0) = v_1(\mu_0) - v_1(\mu_\infty^*).$$ 
A similar argument shows that  $V_2(\mu_0) = v_2(\mu_0) - v_2(\mu_\infty^*)$. 
\end{proof}

\section{Linear Quadratic Control} 
\label{sec:LQG}
This section focuses on the linear-quadratic-Gaussian (LQG) problem. The goal is to find an equilibrium for \eqref{eq:NE} in the context of the LQG problem as an application of the verification theorem, Theorem \ref{thm:vt1}.
Section \ref{sec:LQ-ex0} considers a setting with separable controls and strong coercivity, which covers LQG as a special case and streamlines the presentation for the remainder of the section.
In the subsequent subsections, we present specific LQG problems 
that generalize the illustrative model in Section \ref{sec:intro}.
We solve the infinite-dimensional Master equations directly by exploiting the polynomial structure on the measure space, a calculation that is interesting in its own right. Moreover, the 
results are illuminating for the theory developed in the previous sections.
For instance, the non-uniqueness of the constants $c_i$ can be observed explicitly in \eqref{eq:c1c2}
during the calculation, a phenomenon less intuitive than the non-uniqueness of the value functions $v_i$, which is apparent from the form of the HJB equations.

\subsection{A special setting with separable controls and strong coercivity}    
\label{sec:LQ-ex0}

We now  consider the HJB \eqref{eq:HJB}
in the following setting.

\begin{assumption}
    \label{asm:1}
  The functions   $b$ and $\sigma$ satisfy
        $$b(\mu, x, a) = \bar b(\mu, x) + a, \ 
        \sigma (\mu, x, a) = \bar \sigma(\mu, x),\ \ \forall (\mu,x,a ) \in \mathcal P_2(\R^2)\times \R^2\times \R^2,$$
        where $\bar b: \mathcal P_2(\mathbb R^2) \times \mathbb R^2 \to \mathbb R^2$ and    
        $\bar \sigma: \mathcal P_2(\mathbb R^2) \times \mathbb R^2 \to \mathbb R^{2\times 2}$ are Lipschitz continuous functions, i.e., there exists a positive constant $K $ so that 
       \begin{equation}
\label{eq-Lip}
 |\bar b(\mu_1, x_1) - \bar b(\mu_2, x_2)|^{2} +
        |\bar \sigma(\mu_1, x_1) - \bar \sigma(\mu_2, x_2)|^{2} \le K(\mathbb W_2^{2}(\mu_1, \mu_2) + |x_1 - x_2|^{2})
\end{equation}for any  $(\mu_{1},x_{1}),  (\mu_{2},x_{2})\in \cP_{2}(\R^{2}) \times \R^{2}$.
\end{assumption}

It is obvious that Assumption \ref{asm:1} implies Assumption \ref{asm:bsigma}.

\begin{lemma}
    \label{lem:nonempty}
    If Assumption \ref{asm:1} holds,
    then the set $\mathcal A$ is non-empty, and thus Assumption \ref{asm:control_space} holds.
\end{lemma}
\begin{proof}
    
    Consider the following linear feedback control generated by
    $\bar{\alpha} (\mu, x)= - C x$
    for some constant 
    $C \in \mathbb R$ to be determined. 
    Thanks to \eqref{eq-Lip}, for any $\mu_{0}\in \cP_{2}(\R^{2})$, \eqref{eq:X} has a unique square integrable solution $X$ 
    corresponding to the control driven by $\bar{\alpha}(\mu, x)=-Cx$.
    
    Moreover, for any $(\mu_{1},x_{1}),  (\mu_{2},x_{2})\in \cP_{2}(\R^{2}) \times \R^{2}$, we can compute 
    \begin{align*}
         2& (x_1 - x_2)^\top \left(b(\mu_1, x_1, \bar{\alpha}(\mu_1, x_1)) - b(\mu_2, x_2, \bar{\alpha}(\mu_2, x_2))\right)
        + \left| \bar \sigma(\mu_1, x_1) - \bar \sigma(\mu_2, x_2) \right|^2 \\
        &=  2 (x_1 - x_2)^\top \left(\bar b(\mu_1, x_1) - \bar b(\mu_2, x_2)\right)
        - 2C |x_1 - x_2|^2 + \left| \bar \sigma (\mu_1, x_1) - \bar \sigma(\mu_2, x_2) \right|^2 \\
       & \le -(2C -  (K + 1)) |x_1 - x_2|^2 + K \mathbb W_2^2 (\mu_1, \mu_2).
    \end{align*}
    Therefore, in view of  \cite[Theorem 3.1]{Wang18},  
    with   $C> \frac{2 K +1}{2}$, 
    the law of $X_{t}$ converges to some 
    $\mu \in \cP_{2}(\R^{2})$ in $\bW_2(\R^{2})$. 
    This of course implies that $\mathcal A \neq \emptyset$.
\end{proof}

To proceed, we impose the following strong coercivity assumption  to the cost $\ell_i$.
\begin{assumption} \label{asm:2}
 For each $i = 1, 2$, the function  $\ell_i(\mu, x, a_i)$ is strictly convex in $a_i$ for all $(\mu, x)\in \cP_{2}(\R^{2})\times \R^{2}$ and 
        satisfies
        $$\lim_{|a_i|\to \infty} \frac{\ell_i(\mu, x, a_i)}{|a_i|} = \infty, \quad \forall (\mu, x)\in \cP_{2}(\R^{2})\times \R^{2}.$$
\end{assumption}
We also define Hamiltonian functions $H_i$ by
\begin{equation}
    \label{eq:Hi}
    H_i(\mu, x, p) = \inf_{a_i\in \R} \left\{ 
        pa_i + \ell_i(\mu, x, a_i) \right\}, \quad  \forall (\mu, x, p)\in \cP_{2}(\R^{2})\times \R^{2}\times \R,
\end{equation}
for $i = 1, 2$. 

\begin{corollary}
    \label{cor:vt1}
   Let   Assumptions \ref{asm:ell},  \ref{asm:vi}, \ref{asm:1}, and \ref{asm:2} hold.
    Suppose $(v_1, v_2, c_1, c_2)$ is a solution to the HJB equation, 
    $\forall (\mu, x) \in \mathcal P_2(\mathbb R^2) \times \mathbb R^2,$
    \begin{equation}
    \label{eq:HJB2}
    \begin{cases}
        \displaystyle 
        \int_{\mathbb R^2} \bigg[ 
            H_1 \bigg(\mu, x, \partial_{x_1} 
        \frac{\delta v_1}{\delta \mu}(\mu, x)\bigg)  
        + \frac 1 2 \tr
        \bigg(\bar \sigma \bar \sigma^\top (\mu,x) D_{xx} \frac{\delta v_1}{\delta \mu}
        (\mu, x) \bigg) 
        \\ \displaystyle \hspace{.35in} 
        +
        D_x \frac{\delta v_1}{\delta \mu}(\mu, x) \cdot \bar b(\mu, x) 
        + 
        \partial_{x_2} \frac{\delta v_1}{\delta \mu}(\mu, x) \partial_{p} H_2\bigg(\mu, x, \partial_{x_2} \frac{\delta v_2}{\delta \mu}(\mu, x)\bigg) 
         \bigg] \mu(dx) 
        = c_1, 
        \\ \displaystyle \int_{\mathbb R^2}  \bigg[ 
            H_2\bigg(\mu, x, \partial_{x_2} \frac{\delta v_2}{\delta \mu}(\mu, x)\bigg) 
        + \frac 1 2 \tr
        \bigg(\bar \sigma \bar \sigma^\top (\mu,x) D_{xx} \frac{\delta v_2}{\delta \mu}(\mu, x) \bigg) 
        \\ \displaystyle \hspace{.35in} +
        D_x \frac{\delta v_2}{\delta \mu}(\mu, x) \cdot \bar b(\mu, x) 
        + 
         \partial_{x_1} \frac{\delta v_2}{\delta \mu}(\mu, x) 
        \partial_{p} H_1 \bigg(\mu, x, \partial_{x_1} \frac{\delta v_1}{\delta \mu}(\mu, x) \bigg) \bigg]   \mu(dx)
        = c_2, 
    \end{cases}
\end{equation}
    such that $v_1$ and $v_2$ are fully-$\mathcal C^2$ and 
    satisfy \eqref{eq:chain_condition}
    and \eqref{eq:chain_condition2}. Also assume that there exists a feedback control 
    $\alpha^* \in \mathcal A$
    generated by
     \begin{equation}
        \label{eq:aistar}
        \bar{\alpha}_i^*(\mu, x) = \partial_p H_i \left(\mu, x, \partial_{x_i} \frac{\delta v_i}{\delta \mu}(\mu, x)\right), \quad i =1, 2.
     \end{equation} 
     Suppose 
     $\mathcal L(X^*_t) $ is convergent to
     some unique invariant measure $\mu_\infty^{*}\in \cP_{2}(\R^{2})$
    as $t\to \infty$ with its support $\mathbb R^2$,
    then the value functions $V_i$ of \eqref{eq:NE2} are given by
    $$V_i(\mu_0) = v_i(\mu_0) - v_i(\mu_\infty^{*}), \quad i=1,2.$$
    Moreover, $(c_1, c_2)$ are equal to ergodic constant 
    $(\hat c_1, \hat c_2)$ of \eqref{eq:NE} and 
    $\alpha^*$ is a Nash equilibrium of the auxiliary control problem \eqref{eq:NE2}. 
\end{corollary}
\begin{proof}
    To apply the verification Theorem \ref{thm:vt1}, we need to verify all assumptions.
    Assumption \ref{asm:bsigma} is satisfied by Assumption \ref{asm:1}.
    Assumption \ref{asm:control_space} holds thanks to Lemma \ref{lem:nonempty}. Moreover, Assumptions \ref{asm:1} and \ref{asm:2} together imply Assumption \ref{asm:calH}.

    Using  the definitions of $H_i$ and 
    $\mathcal H_i$ in \eqref{eq:Hi}  and \eqref{eq:calH}, respectively, we can write 
    $$
    \inf_{a_i\in \mathbb R}
    \mathcal H_i(\mu, x, p, Q, (a_i, a_{3-i})) 
    = H(\mu, x, p_i) +
    p\cdot \bar b (\mu, x) + \frac 1 2 
    \text{Tr}(\bar \sigma \bar \sigma^\top (\mu, x) Q) + p_{3-i} a_{3-i},$$ for $i=1,2$. 
    Consequently, under Assumption \ref{asm:2},   
      the minimizer of the function $\mathcal H_i$ over $a_i$ 
    uniquely exists and is given by \eqref{eq:aistar}. 
    Furthermore, the HJB equation \eqref{eq:HJB} can be reduced to 
    \eqref{eq:HJB2}. At last, $\mu_\infty^*$ has full support $\mathbb R^2$ and 
    hence is sufficient to guarantee the uniqueness of the invariant measure of the optimal path $\mathcal L(X_t^*)$ as $t\to\infty$.
    The rest is a consequence of Theorem \ref{thm:vt1}.
\end{proof}

\subsection{A Linear-Quadratic setting with linear cost in measure}
\label{sec:LQ-ex1}

We consider a simplified model with a pair of 
controlled diffusion processes 
\begin{equation}
    \label{eq:X_ex1}
        dX_t = \alpha_t dt + d W_t.
\end{equation}
The cost function for player $i$ ($i=1,2$) is given by
\begin{equation*}
    \hat J_i(\alpha_1, \alpha_2)
    = \lim_{T \to \infty} \frac{1}{T}
    \mathbb E \left[ \int_0^T \left(
        \gamma \mathbb E [ |X_t|^2 ] + (1 - \gamma) |X_{t}|^2
        + r_i |\alpha_{i,t}|^2 
    \right) dt \right],
\end{equation*}
where $\gamma$ and $r_i$ are given constants.
The cost functional of the auxiliary control problem can be accordingly formulated as
\begin{equation} \label{eq:J_ex1}
    J_i(\mu_0, \alpha) = \lim_{T \to \infty} \mathbb{E} 
    \left[ \int_0^T  \left(
        \gamma \mathbb E [ |X_t|^2 ] + (1 - \gamma) |X_{t}|^2
        + r_i |\alpha_{i,t}|^2 
     - \hat c_i \right) dt \right], \quad i = 1, 2,
\end{equation}
where $\hat c_i$ are the ergodic constants to be determined.

The main objective is to identify a Nash equilibrium
$\alpha^* = (\alpha_1^*, \alpha_2^*)$ and ergodic constants
$(\hat c_1, \hat c_2)$ for the game problem defined by \eqref{eq:NE}, 
together with the corresponding value functions $V_i$ 
of the auxiliary control problem defined by \eqref{eq:NE2}.
Here is the main result for this example.

\begin{proposition}
  \label{prop:LQ_ex1}
  Consider the game problem defined through \eqref{eq:X_ex1}-\eqref{eq:J_ex1}. 
  Assume $r_i>0$ for $i=1,2$.
  Then, there exists a unique Nash equilibrium $\alpha^* = (\alpha_1^*, \alpha_2^*)$
  given by the linear feedback control
  \begin{equation*}
      \bar{\alpha}_i^* (\mu, x) = - \frac{1}{\sqrt{r_i}} x_i, \quad i = 1, 2.
  \end{equation*}
  The corresponding ergodic constants $(\hat c_1, \hat c_2)$ are given by
  \begin{equation*}
      \hat c_1 = \sqrt{r_1} + \frac{\sqrt{r_2}}{2}, \quad
      \hat c_2 = \sqrt{r_2} + \frac{\sqrt{r_1}}{2}.
  \end{equation*}
  Moreover, the value functions $V_i$ of \eqref{eq:NE2} are given by
  \begin{equation*}
      V_i(\mu_0) = \int_{\mathbb R^2} \bar V_i(x) \mu_0(dx), 
  \end{equation*}
  where
  \begin{equation*}
      \bar V_i(x) = \sqrt{r_i}x_i^2 + \frac{\sqrt{r_{3-i}}}{2} x_{3-i}^2 - \frac{r_i}{2} 
      - \frac{r_{3-i}}{4}, \quad i = 1, 2.
  \end{equation*}
\end{proposition}

Before we present the proof of Proposition \ref{prop:LQ_ex1}, we make some remarks.
\begin{remark}
    \label{rem:LQ_ex1}
      Note that the cost functionals in both games \eqref{eq:NE} 
      and \eqref{eq:NE2} depend on the parameter $\gamma$.
      However, observing that
      \begin{equation}
        \label{eq:EEX}
        \mathbb E \left[ \gamma \mathbb E[|X_t|^2] + (1 - \gamma) |X_t|^2 \right]
        = \gamma \mathbb E[|X_t|^2] + (1 - \gamma) \mathbb E[|X_t|^2] = \mathbb E[|X_t|^2],
      \end{equation}
      it becomes evident that the Nash equilibrium $\alpha^*$, the ergodic constants $(\hat c_1, \hat c_2)$,
      and the value functions $V_i$ are all independent of $\gamma$, as confirmed by
      Proposition \ref{prop:LQ_ex1}. 
      Consequently, without loss of generality, one could set $\gamma = 0$ in both games
      \eqref{eq:NE} and \eqref{eq:NE2} and prove Proposition \ref{prop:LQ_ex1} 
      using standard dynamic programming principles and HJB equations 
      in the conventional sense, which yields $\bar V_i(x)$, and 
      then recovers $V_i(\mu_0)$ by integrating $\bar V_i$ against $\mu_0$.

      Nevertheless, we retain the parameter $\gamma$ in this example 
      to demonstrate how the probabilistic identity \eqref{eq:EEX} 
      translates into the structure of the associated $\gamma$-dependent Master equation 
      \eqref{eq:HJB2}, ultimately yielding a $\gamma$-independent solution, 
      see \eqref{eq:muI2} in the proof of Proposition \ref{prop:LQ_ex1}.

\end{remark}

\subsubsection{Proof of Proposition \ref{prop:LQ_ex1}}
\label{sec:proof_LQ_ex1}
The problem under consideration fits within the framework of 
 \eqref{eq:X} 
and the running cost function $\ell_i$ in \eqref{eq:cost} are given by 
\begin{equation*}
    b(\mu, x, a) = a,\ \sigma(\mu, x, a) = I_2, \ \text{ and }\ 
     \ell_i(\mu, x, a_i) = F(\mu, x) + r_i |a_i|^2 
\end{equation*}
for some $r_i > 0$, where 
\begin{equation*}
    F^\gamma(\mu, x) = \gamma [\mu]_{I_2} + (1-\gamma)|x|^2, \ \ \forall (\mu, x) \in \mathcal P_2(\mathbb R^2) \times \mathbb R^2.
\end{equation*}
Our goal is to apply Corollary \ref{cor:vt1} to 
find a Nash equilibrium for the game \eqref{eq:NE},
while examine the independence of the parameter $\gamma$ in the solution.

In this case, the Hamiltonian functions $H_i$ of \eqref{eq:Hi} can be computed explicitly: 
\begin{align*}
H_i (\mu, x, p) 
&= \inf_{a_i \in \mathbb R} \left\{ p a_i + \gamma [\mu]_{I_2} + (1-\gamma)|x|^2 + r_i |a_i|^2 \right\} \\
&= F^\gamma(\mu, x) - \frac{|p|^2}{4r_i},
\end{align*}
with infimum attained at $a_i^* = -\frac{p}{2r_i}$.

With the above settings, the HJB equations \eqref{eq:HJB2} can be rewritten as
\begin{equation}
    \label{eq:HJB3}
\begin{cases}
\displaystyle
\int_{\mathbb R^2} \bigg[ F^\gamma(\mu, x) - \frac{1}{4r_1} \left|\partial_{x_1} \frac{\delta v_1}{\delta \mu}(\mu, x)\right|^2 
+ \frac 1 2 \Delta \frac{\delta v_1}{\delta \mu}(\mu, x) \\ \displaystyle \hspace{1.5in}
+ \partial_{x_2} \frac{\delta v_1}{\delta \mu}(\mu, x) \left(-\frac{1}{2r_2} \partial_{x_2} \frac{\delta v_2}{\delta \mu}(\mu, x)\right) \bigg] \mu(dx) = c_1,
\\ \displaystyle
\int_{\mathbb R^2} \bigg[ F^\gamma(\mu, x) - \frac{1}{4r_2} \left|\partial_{x_2} \frac{\delta v_2}{\delta \mu}(\mu, x)\right|^2 
+ \frac 1 2 \Delta \frac{\delta v_2}{\delta \mu}(\mu, x) \\ \displaystyle \hspace{1.5in}
+ \partial_{x_1} \frac{\delta v_2}{\delta \mu}(\mu, x) \left(-\frac{1}{2r_1} \partial_{x_1} \frac{\delta v_1}{\delta \mu}(\mu, x)\right) \bigg] \mu(dx) = c_2.
\end{cases}
\end{equation}

Assume the value functions $v_i$ are of the form
\begin{equation}
    \label{eq:v_ansatz}
v_1(\mu) = [\mu]_{Q}, \qquad v_2(\mu) = [\mu]_{R},
\end{equation}
where $Q, R \in S_2(\mathbb{R})$ are to be determined. 
With this ansatz, the flat derivative satisfies:
\begin{equation}\label{eq:derivative1}
\frac{\delta v_1}{\delta \mu}(\mu,x) = x^\top Q x - [\mu]_{Q}, \qquad 
D_x \frac{\delta v_1}{\delta \mu}(\mu,x) = 2 Q x, \qquad
D_{xx} \frac{\delta v_1}{\delta \mu}(\mu,x) = 2 Q,
\end{equation}
and
\begin{equation}\label{eq:derivative2}
\frac{\delta v_2}{\delta \mu}(\mu,x) = x^\top R x - [\mu]_{R}, \qquad 
D_x \frac{\delta v_2}{\delta \mu}(\mu,x) = 2 R x, \qquad
D_{xx} \frac{\delta v_2}{\delta \mu}(\mu,x) = 2 R,
\end{equation}
Substituting \eqref{eq:derivative1} and \eqref{eq:derivative2} into the HJB equations \eqref{eq:HJB3}, we have
\begin{equation}
    \label{eq:HJB4}
\begin{cases}
\displaystyle
\int_{\mathbb R^2} \left(
\gamma [\mu]_{I_2} + (1-\gamma)|x|^2 - \frac{1}{r_1} (e_1^\top Q x)^2 + 
\tr(Q) -  \frac{2}{r_2} e_2^\top Q x e_2^\top R x \right) \mu(dx) = c_1, \\
\displaystyle
\int_{\mathbb R^2} \left(\gamma [\mu]_{I_2} + (1-\gamma)|x|^2 - \frac{1}{r_2} (e_2^\top R x)^2 + 
\tr(R) -  \frac{2}{r_1} e_1^\top R x e_1^\top Q x \right) \mu(dx) = c_2.
\end{cases}
\end{equation}

Note that 
\begin{equation}
    \label{eq:muI2}
\int_{\mathbb R^2} \left(\gamma [\mu]_{I_2} + (1-\gamma)|x|^2\right) \mu(dx) = \int_{\mathbb R^2} |x|^2 \mu(dx)= [\mu]_{I_2}.
\end{equation}
Therefore, \eqref{eq:HJB4} can be reduced as
\begin{equation}
    \label{eq:HJB5}
\begin{cases}
\displaystyle
[\mu]_{I_2} - \frac{1}{r_1} [\mu]_{(Q e_1)^{\otimes 2}} 
- \frac{2}{r_2} [\mu]_{(Q e_2) \otimes (R e_2)} + \tr(Q) = c_1, \\
\displaystyle
[\mu]_{I_2} - \frac{1}{r_2} [\mu]_{(R e_2)^{\otimes 2}}
- \frac{2}{r_1} [\mu]_{(R e_1) \otimes (Q e_1)} + \tr(R) = c_2.
\end{cases}
\end{equation}
It is worth noting that the above equations are independent to the parameter $\gamma$. 

Note also that $(Qe_2) \otimes (R e_2)$ is not a symmetric matrix, but we can symmetrize it by taking $\frac{1}{2}((Qe_2) \otimes (R e_2) + (R e_2) \otimes (Q e_2))$. 
Indeed,  denoting $\overline{A} = \frac 1 2 (A + A^\top)$ for any $A \in \mathbb R^{2\times 2}$, 
then $x^\top A x = x^\top \overline{A} x$ for any $x \in \mathbb R^2$. Therefore, we can rewrite \eqref{eq:HJB5} as
\begin{equation*}
\begin{cases}
\displaystyle
[\mu]_{I_2} - \frac{1}{r_1} [\mu]_{(Q e_1)^{\otimes 2}} 
- \frac{2}{r_2} [\mu]_{\overline{(Q e_2) \otimes (R e_2)}} + \tr(Q) = c_1, \\[2ex]
\displaystyle
[\mu]_{I_2} - \frac{1}{r_2} [\mu]_{(R e_2)^{\otimes 2}}
- \frac{2}{r_1} [\mu]_{\overline{(R e_1) \otimes (Q e_1)}} + \tr(R) = c_2.
\end{cases}
\end{equation*}

Comparing the coefficients of $[\mu]$ terms and the constant terms, we have the following Riccati system:
\begin{equation}
    \label{eq:Riccati-ex1}
\begin{cases}
\displaystyle
I_2 - \left(\frac{1}{r_1} \right) (Q e_1)^{\otimes 2} - \left(\frac{2}{r_2} \right)\overline{(Q e_2) \otimes (R e_2)} = 0, \\[2ex]
\displaystyle
I_2 - \left(\frac{1}{r_2} \right) (R e_2)^{\otimes 2} - \left(\frac{2}{r_1} \right) \overline{(R e_1) \otimes (Q e_1)} = 0, \\[2ex]
\displaystyle
\tr(Q) = c_1, \\
\displaystyle
\tr(R) = c_2.
\end{cases}
\end{equation}
Component-wise, we can rewrite the first two equations of \eqref{eq:Riccati-ex1} as 
\begin{align*} \begin{cases}
\frac{Q_{11}^2}{r_1} + \frac{2 Q_{12} R_{12}}{r_2} -1 = 0 \\[1ex]
\frac{Q_{12}^2}{r_1} + \frac{2Q_{22} R_{22}}{r_2} -1 =0\\[1ex]
\frac{Q_{11}Q_{12}}{r_{1}} + \frac{Q_{12}R_{22} + Q_{22}R_{12}}{r_{2}} =0\\[1ex]
\frac{R_{12}^{2}}{r_{2}}  + \frac{2R_{11} Q_{11}}{r_{1}} -1 =0\\[1ex]
\frac{R_{22}^{2}}{r_{2}} + \frac{2 R_{12} Q_{12}}{r_{1}} -1 =0\\[1ex]
\frac{R_{12} R_{22}}{r_{2} } + \frac{R_{11}Q_{12} +R_{12} Q_{11}}{r_{1}} =0. 
\end{cases}\end{align*} 

There exist multiple solutions to the above system is given by diagonal matrices
$$
Q = \begin{bmatrix}
    \pm \sqrt{r_{1}} & 0 \\
0 & \frac{\sqrt{r_{2}}}{2}
\end{bmatrix}, \ \ R = \begin{bmatrix}
\frac{\sqrt{r_{1}}}{2} & 0 \\
0 & \pm \sqrt{r_{2}}
\end{bmatrix}.
$$

In view of the last two equations of \eqref{eq:Riccati-ex1}, we have 
\begin{equation}
    \label{eq:c1c2}
    c_1 = Q_{11} + Q_{22} =\pm \sqrt{r_{1}} +  \frac{\sqrt{r_{2}}}{2}, \quad c_2 = R_{11} + R_{22} =\sqrt{r_{2}} \pm  \frac{\sqrt{r_{1}}}{2}.
\end{equation}
This also implies that nonuniqueness of $c_i$ of \eqref{eq:HJB2} is due to the nonuniqueness of the solution to the HJB equation \eqref{eq:Riccati-ex1}.
In addition, using the feedback controls
\begin{equation}
    \label{eq:a_star}
    \bar{\alpha}_1^*(\mu, x) = -\frac{1}{2r_1} \partial_{x_1} \frac{\delta v_1}{\delta \mu}(\mu, x) = -\frac{1}{2r_1} 2 (Q x)_1 
    = \mp \frac{1}{\sqrt{r_1}} x_1,
\end{equation}
and
\begin{equation}
    \label{eq:a_star2}
    \bar{\alpha}_2^*(\mu, x) = -\frac{1}{2r_2} \partial_{x_2} \frac{\delta v_2}{\delta \mu}(\mu, x) = -\frac{1}{2r_2} 2 (R x)_2 
    = \mp\frac{1}{\sqrt{r_2}} x_2,
\end{equation}
the controlled state process $X^{*}_t$ is \begin{align*}
   \begin{cases}
     dX^*_1(t)= \mp\frac{1}{\sqrt{r_1}}   X^*_1(t) dt + dW_1(t), \\
    dX^*_2(t)= \mp\frac{1}{\sqrt{r_2}}   X^*_2(t) dt + dW_2(t).
   \end{cases}
\end{align*} 
One can easily see that, only if the positive solution of $Q$ and $R$ is chosen 
for \eqref{eq:c1c2}, the state process $X^{*}_t$ is ergodic with  $\mathcal L(X^*_t)$ converging to an invariant measure $\pi^* = N(0, \Sigma)\in \mathcal P_2(\R^2)$ in $\bW_2(\R^2)$ as $t \to \infty$, 
where $\Sigma: = \diag(\frac{\sqrt{r_1}}{2}, \frac{\sqrt{r_2}}{2})$. 
This, in turn, implies that \begin{align*}
    \mathcal L(X^*_t, \bar{\alpha}^*(\mathcal L(X^*_t), X^*_t)) \to (\pi^*, \wdt\pi^*) \ \text{ in } \bW_2(\R^4),  
\end{align*} 
where $\wdt \pi^* = N(0, \wdt \Sigma)$ with $\wdt \Sigma := \diag(\frac{1}{2\sqrt{r_1}}, \frac{1}{2\sqrt{r_2}})$.
It is also important to note that the support of $\pi^*$ is $\mathbb R^2$.
Consequently, all conditions of Corollary \ref{cor:vt1} are satisfied. 
Moreover, the pair $(\bar{\alpha}_1^*, \bar{\alpha}_2^*)$ defined in \eqref{eq:a_star}--\eqref{eq:a_star2} 
constitute  a feedback form of the Nash equilibrium, 
the associated ergodic constants of \eqref{eq:NE} 
can be verified as 
\begin{align*}
    \hat J_1(\bar{\alpha}_1^*, \bar{\alpha}_2^*) = \sqrt{r_{1}} +  \frac{\sqrt{r_{2}}}{2}= c_1, 
    \ \ \hat J_2(\bar{\alpha}_1^*, \bar{\alpha}_2^*) = \sqrt{r_{2}} +  \frac{\sqrt{r_{1}}}{2}= c_2,
\end{align*} and the value functions of \eqref{eq:NE2} are given by
\begin{align*}
V_1(\mu_0) = [\mu_0]_{Q} - [\pi^*]_{Q} = [\mu_0]_Q - \frac{r_1}{2} - \frac{r_2}{4}, \quad 
V_2(\mu_0) = [\mu_0]_{R} - [\pi^*]_{R} = [\mu_0]_R - \frac{r_2}{2} - \frac{r_1}{4}.
\end{align*}
If we take $\mu_0 = \delta_x$ for some $x \in \mathbb R^2$, then
\begin{align*}
\bar V_1(x) = V_1(\delta_x) = \sqrt{r_1} x_1^2 + \frac{\sqrt{r_2}}{2} x_2^2 - \frac{r_1}{2} - \frac{r_2}{4}, \quad 
\bar V_2(x) = V_2(\delta_x) = \frac{\sqrt{r_1}}{2} x_1^2 + \sqrt{r_2} x_2^2 - \frac{r_2}{2} - \frac{r_1}{4},
\end{align*}
which 
recovers the standard LQG theory via value functions.

\subsection{An LQ problem with quadratic costs in measure} 
\label{sec:LQ-ex2}
In Section \ref{sec:LQ-ex1}, we analyzed a problem where the structure of the cost functionals allowed for a solution via classical arguments, bypassing the need for Master equations (see Remark \ref{rem:LQ_ex1}). In contrast, this section introduces an example where the cost functionals exhibit a quadratic dependence on the measure variable. This dependence renders the classical approach inapplicable, thereby necessitating the use of the Master equation framework developed earlier.
Consider the controlled diffusion process
\begin{equation}
    \label{eq:X_ex2}
        dX_t = \alpha_t dt + d W_t,
\end{equation}
with the cost function for player $i$ ($i=1,2$) defined as
\begin{equation}
    \label{eq:J_ex2}
    \hat J_i(\alpha_1, \alpha_2)
    = \lim_{T \to \infty} \frac{1}{T}
    \mathbb E \left[ \int_0^T \left( |X_{t}|^2
        + r_i |\alpha_{i,t}|^2 + 
        \left( \mathbb E \left[ \eta_i^\top X_t \right] \right)^2
    \right) dt \right],
\end{equation}
where $r_i \in \mathbb R$ and $\eta_i \in \mathbb R^2$ are given.
The corresponding cost functional of the auxiliary control problem is given by
\begin{equation} \label{eq:J_ex2_aux}
    J_i(\mu_0, \alpha) = \lim_{T \to \infty} \mathbb{E} 
    \left[ \int_0^T  \left( |X_{t}|^2
        + r_i |\alpha_{i,t}|^2 + 
        \left( \mathbb E \left[ \eta_i^\top X_t \right] \right)^2 - \hat c_i \right) dt \right], \quad i = 1, 2,
\end{equation}
where $\hat c_i$ are the ergodic constants to be determined.

The main objective is to identify a Nash equilibrium $\alpha^* = (\alpha_1^*, \alpha_2^*)$ and ergodic constants $(\hat c_1, \hat c_2)$ for the game problem defined by \eqref{eq:NE}, together with the corresponding value functions $V_i$ of the auxiliary control problem defined by \eqref{eq:NE2}.
The main result for this example is provided below via 
the following algebraic Riccati equations,
\begin{equation}
    \label{eq:Riccati}
\begin{cases}
I_2 - r_1^{-1} {(Q_1 e_1)^{\otimes 2}} - 2 r_2^{-1} \overline{(Q_1 e_2) \otimes (Q_2 e_2)} = 0, \\
I_2 - r_2^{-1} {(Q_2 e_2)^{\otimes 2}} - 2 r_1^{-1} \overline{(Q_2 e_1) \otimes (Q_1 e_1)} = 0, \\
\eta_1^{\otimes 2} - 2 r_1^{-1} \overline{(Q_1 e_1) \otimes (R_1 e_1)}
- r_1^{-1} {(R_1 e_1)^{\otimes 2}} 
\\ \hspace{.3in}
- 2 r_2^{-1} \overline{(Q_1 e_2) \otimes (R_2 e_2)} 
- 2 r_2^{-1} \overline{(Q_2 e_2) \otimes (R_1 e_2)} 
- 2 r_2^{-1} \overline{(R_1 e_2) \otimes (R_2 e_2)}
= 0, \\
\eta_2^{\otimes 2} - 2 r_2^{-1} \overline{(Q_2 e_2) \otimes (R_2 e_2)}
- r_2^{-1} {(R_2 e_2)^{\otimes 2}} 
\\  \hspace{.3in}
- 2 r_1^{-1} \overline{(Q_2 e_1) \otimes (R_1 e_1)} 
- 2 r_1^{-1} \overline{(Q_1 e_1) \otimes (R_2 e_1)} 
- 2 r_1^{-1} \overline{(R_1 e_1) \otimes (R_2 e_1)}
= 0, \\
- r_1^{-1}  q_1^\top e_1^{\otimes 2} (Q_1 + R_1) 
- r_2^{-1} q_2^\top e_2^{\otimes 2} (Q_1 + R_1) 
- r_2^{-1} q_1^\top e_2^{\otimes 2} (Q_2 + R_2) = 0, \\
- r_2^{-1} q_2^\top e_2^{\otimes 2} (Q_2 + R_2) 
- r_1^{-1} q_1^\top e_1^{\otimes 2} (Q_2 + R_2) 
- r_1^{-1} q_2^\top e_1^{\otimes 2} (Q_1 + R_1) = 0,\\
c_1 = -  \frac 1 4 r_1^{-1} |e_1^\top q_1|^2
          + \tr(Q_1) - \frac 1 2 r_2^{-1} q_1^\top e_2^{\otimes 2} q_2, \\
c_2 = - \frac 1 4 r_2^{-1} |e_2^\top q_2|^2
          +  \tr(Q_2) - \frac 1 2 r_1^{-1} q_2^\top e_1^{\otimes 2} q_1.
\end{cases}
\end{equation}
\begin{proposition}
  \label{prop:LQ_ex2}
  Consider the game problem defined through \eqref{eq:X_ex2}-\eqref{eq:J_ex2_aux}. 
  Assume $r_i>0$ for $i=1,2$.
  Suppose the Riccati system \eqref{eq:Riccati} admits a solution 
  $(Q_1, Q_2, R_1, R_2, q_1, q_2, c_1, c_2) \in S_2(\mathbb R)^4 \times 
  (\mathbb R^2)^2 \times \mathbb R^2$ such that there exists a positive constant $\e$ so that \begin{align}\label{eq:suff-cond-prop3}
     \lambda_{\min}(\overline Q) - \frac{\e}{2} > \frac{ |{R}|^2}{\e},
  \end{align} where $\overline Q = \frac{Q + Q^\top}{2}$ with \begin{align}\label{eq:Q-R_defn}
    Q = \frac{e_1^{\otimes 2} Q_1}{r_1} + \frac{e_2^{\otimes 2} Q_2}{r_2}, \quad \text{ and } \quad
    R = \frac{e_1^{\otimes 2} R_1}{r_1} + \frac{e_2^{\otimes 2} R_2}{r_2}.
  \end{align}
  Then, \begin{description}
    \item[(i)] the feedback control $\alpha^*$ generated by 
    \begin{equation}
      \label{eq:alpha_ex2}
      \bar{\alpha}_i^* (\mu, x) = - \frac{e_i^\top}{r_i} \left( Q_i x + R_i [\mu]_{\mathbf 1} + q_i/2 \right), 
      \quad i = 1, 2,
  \end{equation}  is a Nash equilibrium for the game problem defined by \eqref{eq:NE}, where $\mathbf 1 = [1, 1]^\top \in \mathbb R^2$;
  \item[(ii)] the associated ergodic constants are given by
  \begin{equation}
      \label{eq:c_ex2}
      \begin{cases}
      \hat c_1 = -  \frac 1 4 r_1^{-1} |e_1^\top q_1|^2
          + \tr(Q_1) - \frac 1 2 r_2^{-1} q_1^\top e_2^{\otimes 2} q_2, \\
      \hat c_2 = - \frac 1 4 r_2^{-1} |e_2^\top q_2|^2
          +  \tr(Q_2) - \frac 1 2 r_1^{-1} q_2^\top e_1^{\otimes 2} q_1;  
      \end{cases}
  \end{equation}
  \item[(iii)] furthermore, the value functions $V_i$ of \eqref{eq:NE2} are given by
  \begin{equation}
      \label{eq:V_ex2}
      V_i(\mu_0) =  v_i(\mu_0) - v_i(\mu_\infty^*), \quad i = 1, 2,
  \end{equation}
  where
  \begin{equation}
      \label{eq:vi_ex2}
      v_i(\mu) = [\mu]_{Q_i} + [\mu]_{q_i} + [\mu]^\top_{\mathbf 1} R_i [\mu]_{\mathbf 1}, \quad i = 1, 2,
  \end{equation} 
  and $\mu_\infty^*$ is the invariant measure of the controlled 
  state process $X_t^*$ under the feedback control $\bar{\alpha}^*$.
  \end{description}
  \end{proposition}

\begin{remark}
    \label{rem:Riccati_ex2}
The key distinction between Proposition \ref{prop:LQ_ex1} and Proposition \ref{prop:LQ_ex2} lies in the structure of the value functions, specifically the choice of ansatz. 
    In the classical LQG setting where the state space is finite-dimensional (e.g., $\mathbb R^2$), the value function is typically a quadratic polynomial in the state variable $x$: 
    $$v_i(x) = x^\top Q_i x + q_i^\top x + r_i, \quad i = 1, 2,$$
    for some $Q_i \in S_2(\mathbb R)$, $q_i \in \mathbb R^2$, and $r_i \in \mathbb R$.
    Counting the degrees of freedom for each player, $Q_i$ contributes 3 unknowns, $q_i$ contributes 2, and $r_i$ contributes 1, totaling 6 unknowns per player. Consequently, solving the game would require a system of 12 algebraic equations to determine these 12 unknowns.

Returning to the example in Proposition \ref{prop:LQ_ex1}, one might expect the value function to be linear in the measure variable, suggesting an ansatz of the form $v_i(\mu) = \int_{\mathbb R^2} f_i(x) \mu(dx)$ for some unknown function $f_i \in C^2(\mathbb R^2; \mathbb R)$. However, since $C^2(\mathbb R^2; \mathbb R)$ is an infinite-dimensional space, solving for $f_i$ directly would generally lead to an intractable system.
To make the problem tractable, we restrict $f_i$ to be a quadratic function in $x$, consistent with the structure of classical LQG problems.
Specifically, one could propose the ansatz
$$v_i(\mu) = [\mu]_{Q_i} + [\mu]_{q_i} + c_i, \quad i = 1, 2.$$
Note that the constant term $c_i$ is irrelevant because the Master equation \eqref{eq:HJB2} involves only the flat derivative of $v_i$. Furthermore, in Proposition \ref{prop:LQ_ex1}, we set $q_i = 0$ to further reduce complexity, leading to the ansatz \eqref{eq:v_ansatz}. Consequently, the resulting Riccati system \eqref{eq:Riccati-ex1} involves only a small number of algebraic equations.

In the context of the current setting \eqref{eq:X_ex2}--\eqref{eq:J_ex2_aux}, one might ask how many equations are required for the Riccati system \eqref{eq:Riccati}.
Based on the intuition that the value functions $v_i(\mu)$ should be quadratic in the measure variable and that $v_i(\delta_x)$ should be a quadratic function in $x \in \mathbb{R}^2$, we adopt the ansatz \eqref{eq:vi_ex2}.
Consequently, for each $i$, $Q_i \in S_2(\mathbb{R})$ contributes 3 unknowns, $q_i \in \mathbb{R}^2$ contributes 2 unknowns, and $R_i \in S_2(\mathbb{R})$ contributes another 3 unknowns.
In total, this results in 16 unknowns for the two players, matching the 16 equations present in \eqref{eq:Riccati}.
It is worth noting that alternative ansatz choices exist, and the selection is crucial for the solvability of the Riccati system;
an ansatz with too few or too many unknowns may lead to an ill-posed system.
Given these complexities, an example with explicit solutions to \eqref{eq:Riccati} is essential.
Accordingly, we present the proof of Proposition \ref{prop:LQ_ex2} below, followed by an example detailing explicit solutions to \eqref{eq:Riccati}.
\hfill $\square$
\end{remark}

\subsubsection{Proof of Proposition \ref{prop:LQ_ex2}}
The proof of Proposition \ref{prop:LQ_ex2} follows by applying Corollary \ref{cor:vt1}.
We first rewrite the HJB equations \eqref{eq:HJB2} for the current setting. In particular, using the ansataz \eqref{eq:vi_ex2}, we derive the algebraic Riccati system \eqref{eq:Riccati}. Next, we verify the conditions of Corollary \ref{cor:vt1} to conclude the proof.

The running cost of \eqref{eq:J_ex2} can be rewritten as
$\ell_i(\mu_t, X_t, \alpha_{i,t})$ with 
\begin{equation*}
    \ell_i(\mu, x, a_i) = |x|^2 + r_i |a_i|^2 + [\mu]_{\eta_i}^2,
\end{equation*}
for $i = 1, 2$.
In this case, the Hamiltonian functions $H_i$ of \eqref{eq:Hi} can 
be computed explicitly.
\begin{align*}
H_i (\mu, x, p)
&= \inf_{a_i \in \mathbb R} \left\{ p a_i + |x|^2 + r_i |a_i|^2 + [\mu]_{\eta_i}^2 \right\} \\
&= |x|^2 + [\mu]_{\eta_i}^2 - \frac{|p|^2}{4r_i},
\end{align*}
with infimum attained at $$a_i^* = -\frac{p}{2r_i}.$$
With the above settings, the HJB equations \eqref{eq:HJB2} can be rewritten as
\begin{equation}
    \label{eq:HJB7}
\begin{cases}
\displaystyle
\int_{\mathbb R^2} \bigg( |x|^2 + [\mu]_{\eta_1}^2 - \frac{1}{4r_1} \left|\partial_{x_1} \frac{\delta v_1}{\delta \mu}(\mu, x)\right|^2 
+ \frac 1 2 \Delta \frac{\delta v_1}{\delta \mu}(\mu, x) \\
\hspace{1.5in} \displaystyle + \partial_{x_2} \frac{\delta v_1}{\delta \mu}(\mu, x) \left(-\frac{1}{2r_2} 
\partial_{x_2} \frac{\delta v_2}{\delta \mu}(\mu, x)\right) \bigg) \mu(dx) = c_1,
\\ \displaystyle
\int_{\mathbb R^2} \bigg( |x|^2 + [\mu]_{\eta_2}^2 - \frac{1}{4r_2    
} \left|\partial_{x_2} \frac{\delta v_2}{\delta \mu}(\mu, x)\right|^2 
+ \frac 1 2 \Delta \frac{\delta v_2}{\delta \mu}(\mu, x) \\
\hspace{1.5in} \displaystyle + \partial_{x_1} \frac{\delta v_2}{\delta \mu}(\mu, x) \left(-\frac{1}{2r_1} 
\partial_{x_1} \frac{\delta v_1}{\delta \mu}(\mu, x)\right) \bigg) \mu(dx) = c_2.
\end{cases}
\end{equation}
and the feedback form of the optimal controls are given by, if the HJB is solved,
\begin{equation}
    \label{eq:a_star3}
    \bar{\alpha}_1^*(\mu, x) = -\frac{1}{2r_1} \partial_{x_1} \frac{\delta v_1}{\delta \mu}(\mu, x), \quad
    \bar{\alpha}_2^*(\mu, x) = -\frac{1}{2r_2} \partial_{x_2} \frac{\delta v_2}{\delta \mu}(\mu, x).
\end{equation}
We propose the following ansatz  by its {\it lift} version of 
the value function    $v_i(\mu)$ given in \eqref{eq:vi_ex2}:
\begin{equation*} 
    v_i(\mu) = \mathbb E [X^\top Q_i X] + \E[X]^\top R_i \E[X] + \E[X]^\top q_i,
\end{equation*} 
where $X=(X_1, X_2)^\top$ is a random variable with distribution $\mu$, 
$Q_i, R_i \in S^2(\mathbb R)$ and $q_i \in \mathbb R^2$ are to be determined.
As mentioned in Remark \ref{rem:Riccati_ex2}, the basic idea of the ansatz is to capture the quadratic polynomial structure of
the value function $v_i(\mu)$ in terms of $\mu$, as well as
the quadratic structure of $v_i(\delta_x)$ in terms of $x$. 

Next, we   derive the related derivatives of $v_i$. 
Note that 
$$\frac{\delta}{\delta \mu} \left( \E[X^\top Q_i X] \right) (\mu, x) = x^\top Q_i x - \E[X^\top Q_i X],$$
$$\frac{\delta}{\delta \mu} \left( \E[X]^\top R_i \E[X] \right) (\mu, x) = 2 \E[X]^\top R_i x - 2 \E[X]^\top R_i \E[X],$$
and 
$$\frac{\delta}{\delta \mu} \left( \E[X]^\top q_i \right) (\mu, x) = x^\top q_i - \E[X]^\top q_i.$$
Therefore we have
\begin{equation}
    \label{eq:derivative_v}
    \frac{\delta v_i}{\delta \mu}(\mu, x) = x^\top Q_i x + 
    (2 \E[X]^\top R_i + q_i^\top) x - C_i,
\end{equation}
where $C_i$ is a constant term given by
$$C_i = \E[X^\top Q_i X] + 2 \E[X]^\top R_i \E[X] + \E[X]^\top q_i,$$
which will be irrelevant at the final form of the Riccati system.
We also need the following derivatives:
\begin{align*}
    D_x \frac{\delta v_i}{\delta \mu}(\mu, x) &= 2 Q_i x + 2 R_i \E[X] + q_i, \\
    D_{xx} \frac{\delta v_i}{\delta \mu}(\mu, x) &= 2 Q_i.
\end{align*}

Combining the above derivatives with the form of the optimal feedback controls in
\eqref{eq:a_star3}, we have
\begin{equation}
    \label{eq:a_star4}
    \begin{cases}
        \bar{\alpha}_1^*(\mu, x) = \displaystyle -\frac{e_1^\top}{2r_1} \left( 2 Q_1 x + 2 R_1 \E[X] + q_1 \right), \\
        \bar{\alpha}_2^*(\mu, x) = \displaystyle-\frac{e_2^\top}{2r_2} \left( 2 Q_2 x + 2 R_2 \E[X] + q_2 \right).
    \end{cases}
\end{equation}

Substituting \eqref{eq:derivative_v} and \eqref{eq:a_star4} 
 into the HJB equations \eqref{eq:HJB7}, we can derive a system of equations for $Q_i, R_i, q_i$ and $c_i$:
\begin{align*}
&\int_{\mathbb R^2} \bigg( |x|^2 + [\mu]_{\eta_1}^2 - \frac{1}{4r_1} \left|e_1^\top \left(2 Q_1 x + 2 R_1 \E[X] + q_1 \right) \right|^2 + \tr(Q_1) \\
&\hspace{1.1in}  -\frac{1}{2r_2}  e_2^\top  \left(2 Q_1 x + 2 R_1 \E[X] + q_1\right) e_2^\top \left( 2 Q_2 x + 2 R_2 \E[X] + q_2\right) \bigg)
\mu(dx) = c_1, \\
&\int_{\mathbb R^2} \bigg( |x|^2 + [\mu]_{\eta_2}^2 - \frac{1}{4r_2} \left|e_2^\top \left(2 Q_2 x + 2 R_2 \E[X] + q_2 \right) \right|^2 + \tr(Q_2) \\
&\hspace{1.1in}  -\frac{1}{2r_1}  e_1^\top  \left(2 Q_2 x + 2 R_2 \E[X] + q_2\right) e_1^\top \left( 2 Q_1 x + 2 R_1 \E[X] + q_1\right) \bigg)
\mu(dx) = c_2.
\end{align*}

Using the facts that 
$$\int_{\mathbb R^2} |x|^2 \mu(dx) = \E[|X|^2], \quad 
[\mu]_{\eta_i}^2 = \mathbb E[X^\top] \eta_i^{\otimes 2} \mathbb E[X],$$
we can further simplify the above equations as follows:
\begin{align*}
&\E [ |X|^2 ] + \mathbb E [X^\top] \eta_1^{\otimes 2} \mathbb E[X] - \frac{1}{4r_1} \E \left[ \left|e_1^\top \left(2 Q_1 X + 2 R_1 \E[X] + q_1 \right) \right|^2 \right] + \tr(Q_1) \\
&\hspace{1in}  -\frac{1}{2r_2}  \E \left[ e_2^\top  \left(2 Q_1 X + 2 R_1 \E[X] + q_1\right)
e_2^\top \left( 2 Q_2 X + 2 R_2 \E[X] + q_2\right) \right] = c_1, \\
&\E [ |X|^2 ] + \mathbb E [X^\top] \eta_2^{\otimes 2} \mathbb E[X] - \frac{1}{4r_2} \E \left[ \left|e_2^\top \left(2 Q_2 X + 2 R_2 \E[X] + q_2 \right) \right|^2 \right] + \tr(Q_2) \\
&\hspace{1in}  -\frac{1}{2r_1}  \E \left[ e_1^\top  \left(2 Q_2 X + 2 R_2 \E[X] + q_2\right) e_1^\top \left( 2 Q_1 X + 2 R_1 \E[X] + q_1\right) \right] = c_2.
\end{align*}
Expanding the above equations, we can further simplify them as follows:
\begin{align*}
\E [ |X|^2 ] &+ \mathbb E [X^\top] \eta_1^{\otimes 2} \mathbb E[X] - 
\frac{1}{r_1} \E \left[ X^\top {(Q_1 e_1)^{\otimes 2}} X \right] 
\\ &  
- \frac{1}{r_1} e_1^\top Q_1 \E[X] e_1^\top \left(2 R_1 \E[X] + q_1\right) 
- \frac{1}{4r_1} \left| e_1^\top (2 R_1 \E[X]  + q_1) \right|^2 + \tr(Q_1) \\
&   
-\frac{2}{r_2} \E \left[ X^\top \overline{(Q_1 e_2) \otimes (Q_2 e_2)} X \right] - \frac{1}{r_2} e_2^\top Q_1 \E[X] e_2^\top \left(2 R_2 \E[X] + q_2\right) 
\\ & 
- \frac{1}{r_2} e_2^\top Q_2 \E[X] e_2^\top \left(2 R_1 \E[X] + q_1 \right) 
- \frac{1}{2 r_2} e_2^\top (2 R_1 \E[X] + q_1) e_2^\top \left(2 R_2 \E[X] + q_2\right) = c_1, 
\end{align*}
and 
\begin{align*}
\E [ |X|^2 ]& + \mathbb E [X^\top] \eta_2^{\otimes 2} \mathbb E[X] - 
\frac{1}{r_2} \E \left[ X^\top {(Q_2 e_2)^{\otimes 2}} X \right]
\\ & 
- \frac{1}{r_2} e_2^\top Q_2 \E[X] e_2^\top \left(2 R_2 \E[X] + q_2\right) 
- \frac{1}{4r_2} \left| e_2^\top (2 R_2 \E[X]  + q_2) \right|^2 + \tr(Q_2) \\
&  
-\frac{2}{r_1} \E \left[ X^\top \overline{(Q_2 e_1) \otimes (Q_1 e_1)} X \right] - \frac{1}{r_1} e_1^\top Q_2 \E[X] e_1^\top \left(2 R_1 \E[X] + q_1\right) 
\\ & 
- \frac{1}{r_1} e_1^\top Q_1 \E[X] e_1^\top \left(2 R_2 \E[X] + q_2 \right) 
- \frac{1}{2 r_1} e_1^\top (2 R_1 \E[X] + q_1) e_1^\top \left(2 R_2 \E[X] + q_2\right) = c_2.
\end{align*}

To obtain the desired Riccati system, we need to match the coefficients of following like terms in the above two equations:
$$
\E [X^\top \Box X], \quad \E[X]^\top \Box \E[X], \quad \text{ and } \quad \E[X]^\top \Box, \quad \Box.
$$
To solve for $(Q_1, Q_2)$, combining like terms of $\E[X^\top \Box X]$, it yields the following equations:
\begin{align*}
& I_2 - r_1^{-1} {(Q_1 e_1)^{\otimes 2}} - 2 r_2^{-1} \overline{(Q_1 e_2) \otimes (Q_2 e_2)} = 0, \\
& I_2 - r_2^{-1} {(Q_2 e_2)^{\otimes 2}} - 2 r_1^{-1} \overline{(Q_2 e_1) \otimes (Q_1 e_1)} = 0.
\end{align*}
For $(R_1, R_2)$, combining like terms of $\E[X]^\top \Box \E[X]$, 
we solve the following equations:
\begin{align*}
&\eta_1^{\otimes 2} - 2 r_1^{-1} \overline{(Q_1 e_1) \otimes (R_1 e_1)}
- r_1^{-1} {(R_1 e_1)^{\otimes 2}} - 
2 r_2^{-1} \overline{(Q_1 e_2) \otimes (R_2 e_2)} 
\\ & \hspace{1.5in}
- 2 r_2^{-1} \overline{(Q_2 e_2) \otimes (R_1 e_2)} 
- 2 r_2^{-1} \overline{(R_1 e_2) \otimes (R_2 e_2)}
= 0, \\
&\eta_2^{\otimes 2} - 2 r_2^{-1} \overline{(Q_2 e_2) \otimes (R_2 e_2)}
- r_2^{-1} {(R_2 e_2)^{\otimes 2}} - 
2 r_1^{-1} \overline{(Q_2 e_1) \otimes (R_1 e_1)} 
\\ & \hspace{1.5in}
- 2 r_1^{-1} \overline{(Q_1 e_1) \otimes (R_2 e_1)} 
- 2 r_1^{-1} \overline{(R_1 e_1) \otimes (R_2 e_1)}
= 0.
\end{align*}
Now, combining like terms of $\Box \E[X] $, we have
\begin{align*}
&- r_1^{-1}  q_1^\top e_1^{\otimes 2} (Q_1 + R_1) 
- r_2^{-1} q_2^\top e_2^{\otimes 2} (Q_1 + R_1) 
- r_2^{-1} q_1^\top e_2^{\otimes 2} (Q_2 + R_2) = 0, \\
&- r_2^{-1} q_2^\top e_2^{\otimes 2} (Q_2 + R_2) 
- r_1^{-1} q_1^\top e_1^{\otimes 2} (Q_2 + R_2) 
- r_1^{-1} q_2^\top e_1^{\otimes 2} (Q_1 + R_1) = 0.
\end{align*}
Lastly, we compare the constant terms or 
the like terms of $\Box$ to obtain the equations for $c_i$: 
\begin{align*}
& -  r_1^{-1} |e_1^\top q_1|^2
+ 4 \tr(Q_1) -
2 r_2^{-1} q_1^\top e_2^{\otimes 2} q_2 = 4 c_1, \\
& -  r_2^{-1} |e_2^\top q_2|^2
+ 4 \tr(Q_2) - 2 r_1^{-1} q_2^\top e_1^{\otimes 2} q_1 = 4 c_2.
\end{align*}
Combining these equations,  we obtain the system of Riccati equations \eqref{eq:Riccati} 
for $(Q_1, Q_2,$ $ R_1, R_2, q_1, q_2)$ and 
the expressions for  $(c_1, c_2)$ given in \eqref{eq:c_ex2}. 

Now let $(Q_1, Q_2,$ $ R_1, R_2, q_1, q_2, c_1, c_2)$ be a solution to \eqref{eq:Riccati} and let $\alpha^*$ be the feedback control given in \eqref{eq:alpha_ex2}. We verify the conditions of Corollary \ref{cor:vt1} to conclude the proof. Obviously, the value functions  $v_i(\mu), i=1,2$ satisfy \eqref{eq:chain_condition}. Using the linear form of $\bar{\alpha}_i^*$ given in \eqref{eq:alpha_ex2}, it is easy to see that \eqref{eq:chain_condition2} holds as well.  

It remains to show that the controlled process $X^*$ corresponding 
to 
$\alpha^*$ of \eqref{eq:alpha_ex2} converges in $\mathbb W_2(\R^2)$ to some invariant measure $\mu^*_\infty$  under  the conditions that \eqref{eq:Riccati} is solvable and that \eqref{eq:suff-cond-prop3} holds. To this end, similar to the calculations in the proof of Lemma \ref{lem:nonempty}, we will   show that there exist positive constants $K_1 > K_2$ so that  \begin{align}\label{eq2-drift-cond}
    (x-y)^\top (\bar{\alpha}^*(\mu,x)-\bar{\alpha}^*(\nu,y)) \le -K_1 |x-y|^2 + K_2\mathbb W_2^2(\mu, \nu),\   \forall x,y\in \R^2 \text{ and } \mu, \nu \in \mathcal P_2(\R^2). 
  \end{align}  
  Then the convergence of $\mathcal L(X_t^*) \to \mu^*_\infty $ in $\mathbb W_2(\R^2)$ 
  follows from \cite[Theorem 3.1]{Wang18}.   
  Using the form of $\bar{\alpha}^*$ in \eqref{eq:alpha_ex2}, the left-hand side of  \eqref{eq2-drift-cond} can be written as \begin{align*}
    -(x-y)^\top Q 
    (x-y)  -  (x-y)^\top R 
    ([\mu]_{\mathbf 1}-[\nu]_{\mathbf 1}). 
  \end{align*} 
   Also note that $$|[\mu]_{\mathbf 1}-[\nu]_{\mathbf 1}| = \bigg|\int \mathbf 1^\top (x-y) \pi(dx, dy)\bigg|  \le \sqrt 2 \bigg(\int |x-y|^2\pi(dx, dy)\bigg)^{1/2}, $$ where $\pi$ is an arbitrary coupling of $\mu$ and $\nu$. Thus $|[\mu]_{\mathbf 1}-[\nu]_{\mathbf 1}|\le \sqrt 2 \mathbb W_2( \mu, \nu)$. 
  Then we have \begin{align*}
    -  (x-y)^\top R([\mu]_{\mathbf 1}-[\nu]_{\mathbf 1}) &\le  |x-y| |{R}| \sqrt 2 \mathbb W_2(\mu, \nu) \\
    &\le \frac{\e}{2} |x-y|^2 + \frac{ |{R}|^2}{\e} \mathbb W_2^2(\mu, \nu), 
  \end{align*} where the second inequality follows from  Young's inequality with $\e$ being the positive constant given in the statement of Proposition \ref{prop:LQ_ex2}. Likewise, since 
  $\overline Q: = \frac{Q + Q^\top}{2}$ is symmetric, we can compute \begin{align*}
    -(x-y)^\top Q(x-y) = - (x-y)^\top \overline{Q} (x-y) \le - \lambda_{\min}(\overline Q) |x-y|^2.
  \end{align*} Using these observations, we have \begin{align*}
     (x-y)^\top (\bar{\alpha}^*(\mu,x)-\bar{\alpha}^*(\nu,y)) \le -\bigg( \lambda_{\min}(\overline Q) - \frac{\e}{2}\bigg)|x-y|^2 + \frac{|{R}|^2}{\e} \mathbb W_2^2(\mu, \nu).
  \end{align*}  
  In view of the assumption \eqref{eq:suff-cond-prop3}, we can choose $K_1 = \lambda_{\min}(\overline Q) - \frac{\e}{2} >0$ and $K_2 = \frac{|{R}|^2}{\e} >0$ so that $K_1 > K_2$. 
  Additionally, the support of $\mu_\infty^*$ is the whole $\R^2$ since $X^*$ is a mean-reverting diffusion process with non-degenerate noise.
  This completes the verification of \eqref{eq2-drift-cond} and thus establishes the existence of an invariant measure $\mu_\infty^*$;  which together with \eqref{eq:alpha_ex2} imply that $\mathcal L(X_t^*, \alpha^*_t) $ converges in $\mathbb W_2(\R^4)$ as well.

Finally, by Corollary \ref{cor:vt1}, we conclude that $\alpha^*$ given in \eqref{eq:alpha_ex2} is a Nash equilibrium
with value functions given in \eqref{eq:V_ex2} and ergodic constants given in \eqref{eq:c_ex2}. 
This completes the proof of Proposition \ref{prop:LQ_ex2}.

\begin{example}
    We illustrate the above result with a numerical example. In \eqref{eq:J_ex2},  let $r_1 =1,  r_2 = 1.5$, and \begin{align*}\eta_1 =\begin{bmatrix}1\\ -0.15\end{bmatrix}, \quad \eta_2 = \begin{bmatrix} 0.2 \\ 1.2 \end{bmatrix}. \end{align*} Solving the Riccati system \eqref{eq:Riccati} numerically, we obtain
    \begin{align*}
        Q_1 = \begin{bmatrix}
        1 & 0 \\
        0 & \frac{\sqrt6}{4}
        \end{bmatrix}, \quad Q_2 = \begin{bmatrix}
        \frac12 & 0 \\
        0 & \frac{\sqrt6}{2}
        \end{bmatrix}, \quad q_1 = \begin{bmatrix} 0 \\ 0 \end{bmatrix}, \quad q_2 = \begin{bmatrix} 0 \\ 0 \end{bmatrix},
    \end{align*}
    and 
    \begin{align*}
        R_1 = \begin{bmatrix}
        0.5611557350 & -0.0769592715 \\
         -0.0769592715 & -0.2135902843
        \end{bmatrix}, \quad R_2 = \begin{bmatrix}
        -0.1828483919 & 0.1117181956 \\
        0.1117181956 &  0.6921092141
        \end{bmatrix}.
    \end{align*}    
    Then we can find  the matrices $Q, R$ as in \eqref{eq:Q-R_defn} and compute the minmum eigenvalue of $Q$ and the norm of $R$ to be  \begin{align*}
        \lambda_{\min}(Q) = 0.7453559925, \quad |R| = 0.4674876586.
    \end{align*} 
    Using $\varepsilon = 0.6611273870$, we can verify that the condition \eqref{eq:suff-cond-prop3} holds:
    \begin{align*}
       \lambda_{\min}(Q) - \frac{\varepsilon}{2} - \frac{|R|^2}{\varepsilon} =  0.0842286055> 0. 
     \end{align*} Therefore, by Proposition \ref{prop:LQ_ex2}, the feedback control $\alpha^*$ given in \eqref{eq:alpha_ex2} is a Nash equilibrium for this game problem. In addition, using \eqref{eq:c_ex2}, the associated ergodic constants are given by \begin{align*}
        \hat c_1 = 1.612372435695794, \quad \hat c_2 = 1.724744871391589;
    \end{align*} the associated value functions $V_i$ of \eqref{eq:NE2} can be computed via \eqref{eq:V_ex2}. 
\end{example}

\begin{remark}
    The sufficient condition \eqref{eq:suff-cond-prop3} for the existence of invariant measure $\mu_\infty^*$ is a rather conservative one, partly due to the fact that it involves a complex interplay between the matrices $Q_1, Q_2, R_1, R_2$. In fact, in the example presented in the next subsection, in which we can derive the explicit forms of these matrices,  we can verify the existence of an invariant measure $\mu_\infty^*$ directly without using \eqref{eq:suff-cond-prop3}. 
    More discussions on the existence of invariant measures is provided in 
    Section \ref{sec:conclusion} as a future research direction.
\end{remark}

\subsubsection{An example with explicit solutions to the Riccati system}
\label{sec:ex_explicit}

We now provide an example 
in which the Riccati system \eqref{eq:Riccati} can be solved explicitly.
\begin{proposition}
  \label{prop:LQ_ex2_example}
  Consider the game problem defined through \eqref{eq:X_ex2}-\eqref{eq:J_ex2_aux} with 
  $r_i>0$ for $i=1,2$. 
  In addition, assume $\eta_1,\eta_2$ satisfy 
  \begin{equation}
    \label{eq:params_special}
  \eta_{1,1} \neq 0, \quad  \eta_{2,2} \neq 0
  ,\quad  \eta_{1,2} = \eta_{2,1} =0.
  \end{equation} 
  Then, the Riccati system \eqref{eq:Riccati} admits a solution 
  $(Q_1, Q_2, R_1, R_2, q_1, q_2) \in S_2(\mathbb R)^4 \times 
  (\mathbb R^2)^2$ given by
  \begin{equation}
    \label{eq:Q1Q2}
Q_1 = \begin{bmatrix}
\sqrt{r_1} & 0 \\
0 & \frac{\sqrt{r_2}}{2}
\end{bmatrix}, \ \ Q_2 = \begin{bmatrix}
\frac{\sqrt{r_1}}{2} & 0 \\
0 & \sqrt{r_2}
\end{bmatrix}, 
\quad 
q_1 = q_2 = \begin{bmatrix}
0 \\ 0
\end{bmatrix}, 
\end{equation}
\begin{equation}
    \label{eq:R1_special}
    R_1 = \begin{bmatrix}
    r_1^{1/2} \Big( -1+  \sqrt{1 + \eta_{1,1}^2} \Big) & 0 \\
    0 & -\frac{r_2^{1/2} \left( -1+  \sqrt{1 + \eta_{2,2}^2} \right)}{2\left( 1 + r_2^{1/2} \left( -1+  \sqrt{1 + \eta_{2,2}^2} \right) \right)}
    \end{bmatrix},
\end{equation}
and
\begin{equation}
    \label{eq:R2_special}
    R_2 = \begin{bmatrix}
    -\frac{r_1^{1/2} \left( -1+  \sqrt{1 + \eta_{1,1}^2} \right)}{2\left( 1 + r_1^{1/2} \left( -1+  \sqrt{1 + \eta_{1,1}^2} \right) \right)} & 0 \\
    0 & r_2^{1/2} \Big( -1+  \sqrt{1 + \eta_{2,2}^2} \Big)
    \end{bmatrix}.
\end{equation}
Accordingly, the feedback form of the Nash equilibrium $\alpha^*$ is given by
\begin{equation*}
    \begin{cases}
        \bar{\alpha}_1^*(\mu, x) 
         = - \frac{1}{\sqrt{r_1}} \left(  x_1 + \left( -1 + \sqrt{1 + \eta_{1,1}^2} \right) [\mu]_{e_1} \right), \\ 
        \bar{\alpha}_2^*(\mu, x) 
        = - \frac{1}{\sqrt{r_2}} \left(  x_2 + \left( -1 + \sqrt{1 + \eta_{2,2}^2} \right) [\mu]_{e_2} \right). 
    \end{cases}
\end{equation*}
Consequently, the ergodic constants $(\hat c_1, \hat c_2)$ are given by
\begin{equation}
    \label{eq:c_i}
    \begin{cases}
        \hat c_1 = \tr(Q_1) = \sqrt{r_1} + \frac{\sqrt{r_2}}{2}, \\
        \hat c_2 = \tr(Q_2) = \frac{\sqrt{r_1}}{2} + \sqrt{r_2}.
    \end{cases}
\end{equation}
Finally, the value function $V_i$ in \eqref{eq:V_ex2} is given by
\begin{equation}
    \label{eq:V_i_ex2}
    V_i(\mu) = [\mu]_{Q_i} + [\mu]_{\mathbf 1}^{\top} R_i [\mu]_{\mathbf 1} - \frac{r_i}{2} 
    - \frac{r_{3-i}}{4}, \quad i = 1, 2,
\end{equation}
with $Q_i$ and $R_i$ given in \eqref{eq:Q1Q2}, \eqref{eq:R1_special} and \eqref{eq:R2_special}.
\end{proposition}

\begin{proof}
  From the first two equations of \eqref{eq:Riccati}, we
can solve for $Q_1$ and $Q_2$, which yields \eqref{eq:Q1Q2}.
Next, we set
$$R_1 = \begin{bmatrix}
    a_1 & d_1 \\
    d_1 & b_1
\end{bmatrix}, 
\quad R_2 = \begin{bmatrix}
    a_2 & d_2 \\
    d_2 & b_2
\end{bmatrix},$$
for some $a_i, b_i, d_i \in \mathbb R$, $i=1,2$.
Substituting them into the third and fourth equations of 
\eqref{eq:Riccati}, 
we have
\begin{align*}
&\begin{bmatrix}
\eta_{1,1}^2 & \eta_{1,1}\eta_{1,2} \\
\eta_{1,1}\eta_{1,2} & \eta_{1,2}^2
\end{bmatrix} - r_1^{-1/2} \begin{bmatrix}
2a_1 & d_1 \\
d_1 & 0
\end{bmatrix} - 
r_1^{-1} \begin{bmatrix}
a_1^2 & a_1 d_1 \\
d_1 a_1 & d_1^2
\end{bmatrix} \\
&\hspace{1in}
- r_2^{-1/2} \begin{bmatrix}0 & \frac{d_2}{2} \\
\frac{d_2}{2} & b_2
\end{bmatrix} - 
r_2^{-1/2} \begin{bmatrix}
0 & d_1 \\
d_1 & 2 b_1
\end{bmatrix} -
r_2^{-1} \begin{bmatrix}
2 d_1 d_2 & d_1 b_2 + b_1 d_2 \\
d_1 b_2 + b_1 d_2 & 2 b_1 b_2
\end{bmatrix} = 0, 
\end{align*}
and
\begin{align*}
&\begin{bmatrix}
\eta_{2,1}^2 & \eta_{2,1}\eta_{2,2} \\
\eta_{2,1}\eta_{2,2} & \eta_{2,2}^2
\end{bmatrix} - r_2^{-1/2} \begin{bmatrix}
    0 & d_2 \\
    d_2 & 2 b_2
\end{bmatrix} - 
r_2^{-1} \begin{bmatrix}
    d_2^2 & d_2 b_2 \\
    b_2 d_2 & b_2^2
\end{bmatrix} \\ 
&\hspace{1in}
- r_1^{-1/2} \begin{bmatrix} 
    a_1 & \frac{d_1}{2} \\
    \frac{d_1}{2} & 0
\end{bmatrix} 
- r_1^{-1/2} \begin{bmatrix}
    2 a_2 & d_2 \\
    d_2 & 0
\end{bmatrix} -
r_1^{-1} \begin{bmatrix}
2 a_1 a_2 & a_1 d_2 + d_1 a_2 \\
a_1 d_2 + d_1 a_2 & 2 d_1 d_2
\end{bmatrix} = 0.
\end{align*}
Component wise, these yield the following six equations:
\begin{equation}
    \label{eq-52}
\begin{cases}
 \eta_{1,1}^2 = 2 r_1^{-1/2} a_1 + r_1^{-1} a_1^2 + 2 r_2^{-1} d_1 d_2  \\
\eta_{1,2}^2 = r_1^{-1} d_1^2 + r_2^{-1/2} (b_2 + 2 b_1) + 2 r_2^{-1} b_1 b_2  \\
\eta_{1,1} \eta_{1,2} = r_1^{-1/2} d_1 + r_1^{-1} a_1 d_1 + r_2^{-1/2} \left( \frac{d_2}{2} + d_1 \right) + r_2^{-1} (d_1 b_2 + b_1 d_2)  \\
\eta_{2,1}^2 = r_2^{-1} d_2^2 + r_1^{-1/2} (a_1 + 2 a_2) + 2 r_1^{-1} a_1 a_2  \\
\eta_{2,2}^2 = 2 r_2^{-1/2} b_2 + r_2^{-1} b_2^2 + 2 r_1^{-1} d_1 d_2  \\
\eta_{2,1} \eta_{2,2} = r_2^{-1/2} d_2 + r_2^{-1} d_2 b_2 + r_1^{-1/2} \left( \frac{d_1}{2} + d_2 \right) + r_1^{-1} (a_1 d_2 + d_1 a_2). 
\end{cases}
\end{equation}
Next, due to the symmetry of the parameters $\eta$ in \eqref{eq:params_special}, 
we look for a solution with $d_1 = d_2 = 0$, which yields
four different solutions to the system \eqref{eq-52} as follows: 
\begin{equation}
    \label{eq:a_b_d_special}
\begin{cases}
    d_1 = d_2 = 0, \\
    a_1 = r_1^{1/2} \left( -1\pm \sqrt{1 + \eta_{1,1}^2} \right), \quad
    b_2 = r_2^{1/2} \left( -1\pm \sqrt{1 + \eta_{2,2}^2} \right), \\
    a_2 = -\frac{a_1}{2(1+ a_1)}, \quad b_1 = -\frac{b_2}{2(1+ b_2)},
\end{cases}
\end{equation} 
Then 
from the fifth and sixth equations of \eqref{eq:Riccati} and noting the fact that the matrices $Q_i, R_i, i=1,2$ are diagonal, $q_1 = (q_{1,1}, q_{1,2})^\top$ and $q_2 = (q_{2,1}, q_{2,2})^\top$ must satisfy 
\begin{equation}
    \label{eq:q_i}
    \begin{cases}
        - r_1^{-1} q_{1,1} [ \sqrt{r_1} + a_1, 0] 
        - r_2^{-1} q_{2,2} [0, \frac{\sqrt{r_2}}{2} + b_1]
        - r_2^{-1} q_{1,2} [0, \sqrt{r_2} + b_2] = 0, \\
        - r_2^{-1} q_{2,2} [0, \sqrt{r_2} + b_2] 
        - r_1^{-1} q_{1,1} [\frac{\sqrt{r_1}}{2} + a_2, 0]
        - r_1^{-1} q_{2,1} [\sqrt{r_1} + a_1, 0] = 0.
    \end{cases}
\end{equation} 
Note that $\sqrt{r_1} + a_1 \neq 0$ and $\sqrt{r_2} + b_2 \neq 0$. 
Thus, from the above equations, we have $q_{1,1} = q_{2,2} =0$, which, in turn, implies $q_{1,2} = q_{2,1} = 0$. In other words, the only solution to \eqref{eq:q_i} is 
\begin{equation*}
    q_{1} = q_{2} = \begin{bmatrix}
        0 \\ 0
    \end{bmatrix}.
\end{equation*} 
Finally, we can derive the constants $c_1$ and $c_2$ from the last two equations of \eqref{eq:Riccati}:
\begin{equation*}
    \begin{cases}
        c_1 = \tr(Q_1) = \sqrt{r_1} + \frac{\sqrt{r_2}}{2}, \\
        c_2 = \tr(Q_2) = \frac{\sqrt{r_1}}{2} + \sqrt{r_2}.
    \end{cases}
\end{equation*}

Next, we investigate the optimal path and screen out the optimal strategies that  do not 
yield an ergodic optimal path. To this end, 
using optimal feedback controls in \eqref{eq:alpha_ex2}, we can derive the following
optimal strategies:
\begin{equation}
    \label{eq:a_star5}
    \begin{cases}
        \bar{\alpha}_1^*(\mu, x) = \displaystyle -\frac{1}{r_1} \left( \sqrt{r_1} x_1 + a_1 \E[X_1]  \right) = - \frac{1}{r_1} \left( \sqrt{r_1} x_1 + a_1 [\mu]_{e_1} \right), \\ 
        \bar{\alpha}_2^*(\mu, x) = \displaystyle -\frac{1}{r_2} \left( \sqrt{r_2} x_2 + b_2 \E[X_2] \right)= - \frac{1}{r_2} \left( \sqrt{r_2} x_2 + b_2 [\mu]_{e_2} \right). 
    \end{cases}
\end{equation}
Accordingly, the optimal path $X:=X^*$ under the feedback control can be written as
\begin{equation}
    \label{eq:X_star_ex2_pf}
    \begin{cases}
        dX_1(t) = - \frac{1}{r_1} \left( \sqrt{r_1} X_1(t) + a_1 \E[X_1(t)]   \right) dt + dW_1(t), \\
        dX_2(t) = - \frac{1}{r_2} \left( \sqrt{r_2} X_2(t) + b_2 \E[X_2(t)]   \right) dt + dW_2(t).
    \end{cases}
\end{equation}
Denoting $m_i(t) := \E[X_i(t)]$, $i=1,2$, then it follows from \eqref{eq:X_star_ex2_pf} that
\begin{equation*}
    \begin{cases}
        dm_1(t) = - \frac{1}{\sqrt{r_1}} \left(  \pm \sqrt{1 + \eta_{1,1}^2} \right) m_1(t) dt , \\ 
        dm_2(t) = - \frac{1}{\sqrt{r_2}} \left( \pm \sqrt{1 + \eta_{2,2}^2} \right) m_2(t) dt. 
    \end{cases}
\end{equation*} 
Solving these equations, we have\begin{align}\label{eq:m_i}
    m_1(t) = m_1(0) \exp\bigg\{\mp t\sqrt{(1+\eta_{1,1}^2)/r_1 }\bigg\}, \quad
    m_2(t) = m_2(0) \exp\bigg\{\mp t \sqrt{(1+\eta_{2,2}^2)/ {r_2} }\bigg\}.
\end{align}
To ensure that the optimal path $X^*$ is ergodic, we need to choose the parameters such that
the mean process $m_i(t)$, $i=1,2$ converges to a finite constant as $t \to \infty$.
This requires us to select the positive square roots for $a_1$ and $b_2$ in \eqref{eq:a_b_d_special}. 
Therefore, the matrices $R_1, R_2$ must be given as in \eqref{eq:R1_special} and \eqref{eq:R2_special}.

Plugging the expressions of $m_i(t)$ in \eqref{eq:m_i} into \eqref{eq:X_star_ex2_pf}, we can  derive the explicit forms of the optimal paths as follows: \begin{align*}
    X_1(t) = (X_1(0) -m_1(0)) e^{-\frac{t}{\sqrt{r_1}}} + m_1(0) e^{- t \sqrt{(1+\eta_{1,1}^2)/r_1 }} + \int_0^t e^{-\frac{t-s}{\sqrt{r_1}}} dW_1(s), \\
    X_2(t) = (X_2(0) -m_2(0)) e^{-\frac{t}{\sqrt{r_2}}} + m_2(0) e^{- t \sqrt{(1+\eta_{2,2}^2)/r_2 }} + \int_0^t e^{-\frac{t-s}{\sqrt{r_2}}} dW_2(s).
\end{align*}
Then it follows that $\mathcal L(X(t))$ converges to 
 the invariance measure $\pi^*$, where 
\begin{equation}
    \label{eq:pi_star}
    \pi^* = \mathcal N \left(0, \begin{bmatrix}
        \frac{\sqrt{r_1}}{2} & 0 \\
        0 & \frac{\sqrt{r_2}}{2}
    \end{bmatrix} \right),
\end{equation} 
which has its support on entire $\mathbb R^2$. 
This together with the   forms of $\ba^*(\mu, x)$ of 
\eqref{eq:a_star5} imply that $$\mathcal L(X(t), \ba^*(\mathcal L(X(t)), X(t))) \to (\pi^*, \wdt\pi^*) \ \text{ in }\ \mathbb W_2(\R^4),$$ where $\wdt \pi^* = N(0, \wdt \Sigma)$ with $\wdt \Sigma := \diag(\frac{1}{2\sqrt{r_1}}, \frac{1}{2\sqrt{r_2}})$. 
Thus all conditions of Corollary \ref{cor:vt1} are satisfied. 
Therefore $\ba^*(\mu, x)$ of \eqref{eq:a_star5} form a Nash equilibrium, 
the constants $c_1, c_2$ of \eqref{eq:c_i} are the ergodic constants, and the corresponding value functions $V_i, i=1,2$ of \eqref{eq:NE2} can be expressed as
\begin{equation*}
        V_i(\mu) = \mathbb E_\mu \left[ X^\top Q_i X \right] + \E_\mu[X]^\top R_i \E_\mu[X] - 
        \mathbb E_{\pi^*} [ X^\top Q_i X ], \quad i=1,2, 
\end{equation*} 
with $Q_i$, $R_i$, and $\pi^*$ given in \eqref{eq:Q1Q2}, \eqref{eq:R1_special}, \eqref{eq:R2_special}, and \eqref{eq:pi_star}, respectively. 
This yields \eqref{eq:V_i_ex2}.
\end{proof}

\section{Conclusion}
\label{sec:conclusion}
In this paper, we have studied a class of nonzero-sum stochastic differential games with McKean-Vlasov dynamics and ergodic cost criteria.
By analyzing the corresponding system of fully nonlinear infinite dimensional HJB equations, we have established the existence of a Nash equilibrium and characterized the value functions. 
For linear-quadratic games, we have derived explicit solutions to the Riccati system and provided examples with closed-form expressions for the Nash equilibrium strategies and value functions. 
Future research directions include extending the current framework to more general dynamics and cost structures, as well as exploring numerical methods for solving the associated HJB systems. 

Some remarks are in order regarding the insights of explicit solutions of the examples 
in Section \ref{sec:LQG} in connection to the theoretical results in the paper. 
One issue pertains to the connection between the Master equations \eqref{eq:HJB2} 
and the cell problems arising in the homogenization of Hamilton-Jacobi equations; 
see, e.g., \cite{LPV87, Tra21} and the references therein. 
It is standard to show that the cell problem admits a unique constant $c$
such that the cell problem admits a viscosity solution, 
see Theorem 4.2 in \cite{Tra21} for instance.
In that context, one might naturally investigate solution uniqueness through the 
lens of viscosity solutions on Wasserstein space. However, as demonstrated by 
the explicit examples in Section \ref{sec:LQ-ex1}, the uniqueness of 
the constants $c_i$ ($i=1,2$) does not hold, see the equation \eqref{eq:c1c2}. 
Therefore, our verification theorem (Theorem \ref{thm:vt1}) plays a crucial role here: 
it employs ergodic stability conditions to uniquely identify the constants $c_i$ 
that correspond to the ergodic constants $\hat c_i$ in the Nash equilibrium, 
thereby resolving the inherent non-uniqueness in the Master equation framework.

Another interesting behavior cast by the example in Section \ref{sec:LQ-ex2} is that, 
the players are explicitly coupled through the cost function \eqref{eq:J_intro}. 
Interestingly, our analysis reveals that this coupling effect vanishes at the Nash equilibrium; see Proposition \ref{prop:LQ_ex2_example}. 
A possible explanation is that the coupling term is symmetric for both players---akin to the interaction in mean-field games---which may lead to a cancellation effect at equilibrium. 
However, since the game is generally asymmetric (unless $r_1 = r_2$ and $\eta_1 = \eta_2$), this decoupling phenomenon warrants further investigation in future research.

Finally, concerning the sufficient condition \eqref{eq:suff-cond-prop3} in Section \ref{sec:LQ-ex2} 
for the existence of an invariant measure $\mu_\infty^*$, 
we note that exploiting the specific structure of the model, particularly \eqref{eq:alpha_ex2}, 
we can show that $m(t) := \mathbb E[X_t^*]\in \R^2$ satisfies the following ODE:
\begin{align*}
    \frac{d m(t)}{dt} = - \big(  Q +R \big) m(t)  - q,
\end{align*} 
where the matrices $Q, R$ are defined in \eqref{eq:Q-R_defn}  and $q = \frac{e_1^{\otimes 2}  q_1}{2 r_1} + \frac{e_2^{\otimes 2} q_2}{2 r_2}$.
Consequently, if the matrix $Q+R$ is nonsingular with all eigenvalues having positive real parts, then $m(t)$ converges to some limit $m_\infty\in \R^2$ as $t \to \infty$. This condition is weaker than \eqref{eq:suff-cond-prop3}. 
However, establishing the convergence of the distribution of $X_t^*$ in $\mathbb W_2(\R^2)$ requires further analysis.
The example in Section \ref{sec:ex_explicit} demonstrates this approach.
We also note that the existence of an invariant measure $\mu_\infty^*$ can be studied through Lyapunov functions and coupling methods, as in \cite{EberGZ:19}, which we defer to future work.

\bibliographystyle{plain}
\hypersetup{linkcolor=blue,citecolor=blue,urlcolor=blue} 
\def\cprime{$'$}

\end{document}